\documentclass[a4paper,12pt]{amsart}
\usepackage{amsmath}
\usepackage{amsfonts}
\usepackage{latexsym, amssymb, amsmath, amsthm, bm}
\usepackage[colorlinks,linkcolor=blue,citecolor=blue]{hyperref}
\usepackage[all]{xy}
\usepackage{enumitem}
\usepackage{extarrows}
\usepackage{listings}
\usepackage{marvosym}
\usepackage{tensor}
\allowdisplaybreaks[4]

\newtheorem{thm}{Theorem}[section]

\newtheorem{prop}[thm]{Proposition}
\newtheorem{coro}[thm]{Corollary}
\newtheorem{defn}[thm]{Definition}

\newtheorem{rmk}[thm]{Remark}

\begin{document}
	\title[Factorization of quasitriangular structures of smash biproduct...]{Factorization of quasitriangular structures of smash biproduct bialgebras}
	
	\author[F. Wang]{Fujun Wang}
	\address{Shanghai Center for Mathematical Sciences, Fudan University, Shanghai 200438, China}
	\email{18110840005@fudan.edu.cn}

	\date{}

	\begin{abstract}
		In this paper, we consider the factorization and reconstruction of quasitriangular structures of smash biproduct bialgebras. Let $A{_\tau\times_\sigma}B$ be a smash biproduct bialgebra. Under condition that $\sigma$ is right conormal, we prove that $A{_\tau\times_\sigma}B$ is quasitriangular if and only if there exists a set of normalized elements $W\in B\otimes B$, $X\in A\otimes B$, $Y\in B\otimes A$ and $Z\in A\otimes A$ satisfying a certain series of identities. In this case, the quasitriangular structure of $A{_\tau\times_\sigma}B$ is given as $\sum \tensor{Z}{^1_{\tau_1\tau_2}}\tensor{\bar{X}}{^1_{\tau_3}}X^1\otimes W^1Y^1\otimes Z^2\tensor{Y}{^2_{\sigma_1\sigma_2}}\epsilon_B(1_{B\tau_1\sigma_2}\tensor{\bar{X}}{^2_{\sigma_1}})\otimes1_{B\tau_2}1_{B\tau_3}X^2W^2$. Our result generalizes the similar results for Radford's biproduct Hopf algebras studied by L. Zhao and W. Zhao, for bicrossproduct Hopf algebras studied by Zhao, Wang and Jiao, and for the dual Hopf algebras of double cross product Hopf algebras studied by Jiao.
	\end{abstract}
	
	

	\maketitle

	\section{Introduction}
	
	Quasitriangular Hopf algebras have been studied for decades since the concept was introduced by Drinfeld. It has wide applications in various fields such as solutions to quantum Yang-Baxter equations \cite{D,Ma}, invariants of knots \cite{R2,K&R,T}, classification of finite-dimensional Hopf algebras \cite{E&G,E&G2,G}, etc.
	
	Smash biproducts are considered independently by Bespalov and Dranbant in \cite{B&D,B&D2}, by Schauenburg in \cite{Sch} and by Caenepeel, Ion, Militaru and Zhu in \cite{C&I&M&Z}. This construction provides a universal way to construct bialgebras (or Hopf algebras) in terms of two linear maps of two algebras which meanwhile are also coalgebras. Moreover smash biproducts bear a characterization by using projections and injections combining the universal properties of smash products and smash coproducts. Based on this characterization, Radford's biproducts, bicrossproducts, and double cross products can be described as smash biproducts of different types.
	
	It's natural to ask when a smash biproduct Hopf algebra, say $A{_\tau\times_\sigma}B$, admits a quasitriangular structure, especially with respect to $A$ and $B$. From this point, the factorization of quasitriangular structures has been investigated for different constructions of Hopf algebras. In \cite{Z&W&J}, Zhao, Wang and Jiao first considered the case of bicrossproduct Hopf algebras and gave an equivalent condition for a bicrossproduct Hopf algebra to be quasitriangular in \cite[Theorem 2.7]{Z&W&J}. Later in \cite{J} Jiao considered the case of a class of \textit{T}-smash product Hopf algebras. In notation of $A{_\tau\times_\sigma}B$, they are smash biproduct Hopf algebras for $\sigma=\textit{T}$ and $\tau$ the usual flipping map. Under condition of \textit{T} being right conormal (cf. \cite[Definition 2.3]{J}), an equivalent condition was given to make a \textit{T}-smash product Hopf algebra quasitriangular in \cite[Theorem 3.7]{J}. In 2007, L. Zhao and W. Zhao considered the case for Radford's biproduct Hopf algebras in \cite{Z&Z}. For this important construction, an equivalent condition to make a Radford's biproduct Hopf algebra quasitriangular was given in \cite[Theorem 1]{Z&Z}. Comparing to \cite{J} in a dual way, $\omega$-smash coproduct Hopf algebras are considered in \cite{J2}. These Hopf algebras are the case of $\sigma$ the usual flipping map and $\tau=\omega$. Also an equivalent condition was given in \cite[Theorem 3.1]{J2} to make an $\omega$-smash coproduct Hopf algebra quasitriangular. In 2010, Ma and Wang considered the twisted tensor biproduct Hopf algebras in \cite{M&W}. Under favorable conditions (cf. the statements in the beginning of Section 3 of \cite{M&W}), an equivalent condition was given in \cite[Theorem 3.13]{M&W}. For these distinct constructions, one could see there is much similarity for the reconstruction formulas in the theorems mentioned above, mainly on the positions of those factor elements. This motivates us to develop a general theory.
	
	For a quasitriangular structure $\mathcal{R}=\sum\mathcal{R}^1\otimes \mathcal{R}^2\otimes \mathcal{R}^3\otimes \mathcal{R}^4\in(A{_\tau\times_\sigma}B)\otimes(A{_\tau\times_\sigma}B)$ of a smash biproduct bialgebra $A{_\tau\times_\sigma}B$, let $W\in B\otimes B$, $X\in A\otimes B$, $Y\in B\otimes A$ and $Z\in A\otimes A$ be the  factor elements of $\mathcal{R}$. We show in Proposition \ref{thm:3.4} that if $\sigma$ is right conormal, then $\mathcal{R}$ can be expressed in terms of the factor elements. The expression is
	\begin{align*}
		\mathcal{R}=\sum \tensor{Z}{^1_{\tau_1\tau_2}}\tensor{\bar{X}}{^1_{\tau_3}}X^1\otimes W^1Y^1\otimes Z^2\tensor{Y}{^2_{\sigma_1\sigma_2}}\epsilon_B(1_{B\tau_1\sigma_2}\tensor{\bar{X}}{^2_{\sigma_1}})\otimes1_{B\tau_2}1_{B\tau_3}X^2W^2.
	\end{align*}
	Then in Theorem \ref{thm:3.6}, we prove that for some normalized elements $W\in B\otimes B$, $X\in A\otimes B$, $Y\in B\otimes A$ and $Z\in A\otimes A$, the reconstruction formula above will be a quasitriangular structure of $A{_\tau\times_\sigma}B$ supposing that these $W$, $X$, $Y$ and $Z$ satisfy a series of identities listed in Proposition \ref{thm:3.5}. As a summary, we get an equivalence theorem for a smash biproduct bialgebra $A{_\tau\times_\sigma}B$ for which $\sigma$ is right conormal to be quasitriangular. Our main result is as follows (see Theorem \ref{thm:3.7}):
	\begin{itemize}[leftmargin=12pt]
		\item Let $A{_\tau\times_\sigma}B$ be a smash biproduct bialgebra. Under condition that $\sigma$ is right conormal, $A{_\tau\times_\sigma}B$ admits a quasitriangular structure if and only if there exist normalized elements $W=\sum W^1\otimes W^2\in B\otimes B$, $X=\sum X^1\otimes X^2\in A\otimes B$, $Y=\sum Y^1\otimes Y^2\in B\otimes A$ and $Z=\sum Z^1\otimes Z^2\in A\otimes A$ satisfying $C1)-C19)$ (see Proposition \ref{thm:3.5}). In this case, the quasitriangular structure is given as $\alpha=\sum \tensor{Z}{^1_{\tau_1\tau_2}}\tensor{\bar{X}}{^1_{\tau_3}}X^1\otimes W^1Y^1\otimes Z^2\tensor{Y}{^2_{\sigma_1\sigma_2}}\epsilon_B(1_{B\tau_1\sigma_2}\tensor{\bar{X}}{^2_{\sigma_1}})\otimes1_{B\tau_2}1_{B\tau_3}X^2W^2$.
	\end{itemize}
	
	This paper is organized as follows. In Section \ref{Sec:2}, we recall the definitions and notations of quasitriangular structures of bialgebras and smash biproducts. Also we recall the definitions of normal and conormal conditions. Some useful results about smash biproducts are also presented. In Section \ref{Sec:3} we first recall the definition of factor elements for a quasitriangular structure. Then we establish the factorization theory of smash biproduct bialgebras $A{_\tau\times_\sigma}B$ for which $\sigma$ is right conormal. In Section \ref{Sec:4} we study smash biproduct bialgebras which satisfy an extra normal or conormal condition besides the original right conormal condition. We get corollaries \ref{thm:4.1}, \ref{thm:4.2} and \ref{thm:4.3}. When the smash biproduct bialgebras are moreover Hopf algebras, these corollaries turn out to be equivalent to the results in \cite{M&W,Z&W&J,J2}, respectively. Thus our theory unifies these distinct theories.

	\section{Preliminaries}\label{Sec:2}
	
	\subsection{Quasitriangular bialgebras}
	In this paper, all objects and morphisms are assumed to be over a field $\Bbbk$.
	
	An algebra $(A,m_A,\eta_A)$ is a $\Bbbk$-vector space $A$ equipped with an associative multiplication $m_A:A\otimes A\rightarrow A\enspace(\sum_ia_i\otimes b_i\mapsto\sum_ia_ib_i)$ and a unit $\eta_A:\Bbbk\rightarrow A\enspace(k\mapsto k1_A)$. If no risk of confusion arises, it will be simply denoted by $A$. An algebra map $f:A\rightarrow A'$ is a $\Bbbk$-linear map of two algebras which preserves the multiplications and units. Dually a coalgebra $(C,\varDelta_C,\epsilon_C)$ is a $\Bbbk$-vector space $C$ equipped with a coassociative comultiplication $\varDelta_C:C\rightarrow C\otimes C\enspace(c\mapsto\sum c_{(1)}\otimes c_{(2)})$ and a counit $\epsilon_C:C\rightarrow\Bbbk \enspace(c\mapsto \epsilon_C(c))$. Here we use Sweedler's notation for comultiplications. Similarly it will be simply denoted by $C$. A coalgebra map $g:C\rightarrow C'$ is a $\Bbbk$-linear map of two coalgebras which preserves the comultiplications and counits.
	
	An algebra $A$ is called counital if it's equipped with a map $f:A\rightarrow\Bbbk$ such that $f$ is an algebra map. A coalgebra $C$ is called unital if it's equipped with a map $g:\Bbbk\rightarrow C$ such that $g$ is a coalgebra map.
	
	A bialgebra $(B,m_B,\eta_B,\varDelta_B,\epsilon_B)$ is an algebra $(B,m_B,\eta_B)$ and also a coalgebra $(B,\varDelta_B,\epsilon_B)$ such that both $m_B$ and $\eta_B$ are coalgebra maps, or equivalently both $\varDelta_B$ and $\epsilon_B$ are algebra maps. For simplicity this bialgebra will be denoted by $B$. Of course a bialgebra $B$ is a counital algebra via $\epsilon_B$ and a unital coalgebra via $\eta_B$. For more details of bialgebras, we suggest \cite{S,M,R}.
	
	We recall the definition of quasitriangular structures of bialgebras as in \cite{R}. Note that in \cite{M} and lots of other literature, quasitriangular structures are usually considered for Hopf algebras.
	\begin{defn}
		[\protect{\cite[Definition 12.2.3]{R}}] A quasitriangular bialgebra is a pair $(H,\mathcal{R})$ where $H$ is a bialgebra and $\mathcal{R}=\sum\mathcal{R}^1\otimes\mathcal{R}^2\in H\otimes H$ such that the following conditions hold:
		\begin{align*}
			&(QT1)\quad(\epsilon_H\otimes id_H)(\mathcal{R})=1_H=(id_H\otimes \epsilon_H)(\mathcal{R}),\hspace{16em}\\
			&(QT2)\quad(\varDelta_H\otimes id_H)(\mathcal{R})=\mathcal{R}_{13}\mathcal{R}_{23},\\
			&(QT3)\quad(id_H\otimes\varDelta_H)(\mathcal{R})=\mathcal{R}_{13}\mathcal{R}_{12},\\
			&(QT4)\quad\mathcal{R}\varDelta_H(h)=\varDelta^{cop}_H(h)\mathcal{R}\quad\forall h\in H.
		\end{align*}
	In this case, $H$ is called quasitriangular and $\mathcal{R}$ is a quasitriangular structure of $H$. Moreover if $H$ is a Hopf algebra, $(H,\mathcal{R})$ is called a quasitriangular Hopf algebra.
	\end{defn}
	\begin{rmk}
		(i) According to \cite[Theorem 12.2.8]{R}, the definition of quasitriangular Hopf algebras above is equivalent to \cite[Definition 10.1.5]{M}. The latter is the usual definition of quasitriangular Hopf algebras.\\
		(ii) There are many quasitriangular bialgebras that are not Hopf algebras. For example, let $S$ be a monoid, then $\Bbbk S$ is an algebra. With the comultiplication $\varDelta(s)=s\otimes s$ and counit $\epsilon(s)=1$ for $\forall s\in S$, $(\Bbbk S, \varDelta, \epsilon)$ is a bialgebra. But it's not a Hopf algebra since $S$ is not a group. Because $\Bbbk S$ is cocommutative, $\mathcal{R}=1_S\otimes1_S$ is always a quasitriangular structure of the bialgebra $\Bbbk S$. Specially consider $S=(\mathbb{Z}_6,\cdot)$. Then $\mathcal{R}=\underline{3}\otimes\underline{3}$ is a non-trivial quasitriangular structure and not invertible. Here $\underline{m}$ represents the element of $\mathbb{Z}_6$. 
	\end{rmk}

	\subsection{Smash biproduct bialgebras}
	Let $A$ and $B$ be two algebras. Suppose that there is a $\Bbbk$-linear map $\sigma:B\otimes A\rightarrow A\otimes B\quad(b\otimes a\mapsto\sum a_{\sigma}\otimes b_{\sigma})$. Then on $A\otimes B$, we can consider the multiplication
	\begin{align*}
		\bm{m}=(m_A\otimes m_B)\circ(id_A\otimes\sigma\otimes id_B):A\otimes B\otimes A\otimes B&\rightarrow A\otimes B\\
		a\otimes b\otimes a'\otimes b'&\mapsto\sum aa'_{\sigma}\otimes b_{\sigma}b'.
	\end{align*}
	If this multiplication is associative and $\eta_A\otimes\eta_B$ is a unit, we call this algebra a smash product of $A$ and $B$, denoted by $A\#_\sigma B$. For $A\#_\sigma B$, we have the following identities:
	\begin{align*}
		&\sum1_{A\sigma}\otimes b_{\sigma}=1_{A}\otimes b,\\
		&\sum a_\sigma\otimes 1_{B\sigma}=a\otimes1_B,\\
		(1.1)\quad&\sum(aa')_\sigma\otimes b_\sigma=\sum a_{\sigma_1}a'_{\sigma_2}\otimes b_{\sigma_1\sigma_2},\\
		(1.2)\quad&\sum a_\sigma\otimes (bb')_\sigma=\sum a_{\sigma_1\sigma_2}\otimes b_{\sigma_2}b'_{\sigma_1}
	\end{align*}
	for $\forall a,a'\in A$, $b,b'\in B$, where $\sigma_i$ is a copy of $\sigma$ for $i=1,2$ and so on.
	
	Dually, let $C$ and $D$ be two coalgebras. Suppose that there is a $\Bbbk$-linear map $\tau:C\otimes D\rightarrow D\otimes C\quad(c\otimes d\mapsto\sum d_{\tau}\otimes c_{\tau})$. Then on $C\otimes D$, we can consider the comultiplication
	\begin{align*}
		\bm{\varDelta}=(id_C\otimes\tau\otimes id_D)\circ(\varDelta_C\otimes \varDelta_D):C\otimes D&\rightarrow C\otimes D\otimes C\otimes D\\
		c\otimes d&\mapsto\sum c_{(1)}\otimes d_{(1)\tau}\otimes c_{(2)\tau}\otimes d_{(2)}.
	\end{align*}
	If this comultiplication is coassociative and $\epsilon_C\otimes\epsilon_D$ is a counit, we call this coalgebra a smash coproduct of $C$ and $D$, denoted by $C{_\tau\ltimes} D$. For $C{_\tau\ltimes} D$, we have the following identities:
	\begin{align*}
		&\sum\epsilon_D(d_\tau)\otimes c_\tau=\epsilon_D(d)c,\\
		&\sum d_\tau\otimes\epsilon_C(c_\tau)=\epsilon_C(c)d,\\
		(1.3)\quad&\sum d_{\tau(1)}\otimes d_{\tau(2)}\otimes c_{\tau}=\sum d_{(1)\tau_1}\otimes d_{(2)\tau_2}\otimes c_{\tau_1\tau_2},\\
		(1.4)\quad&\sum d_{\tau}\otimes c_{\tau(1)}\otimes c_{\tau(2)}=\sum d_{\tau_2\tau_1}\otimes c_{(1)\tau_1}\otimes c_{(2)\tau_2}
	\end{align*}
	for $\forall c\in C$, $d\in D$, where $\tau_i$ is a copy of $\tau$ for $i=1,2$ and so on.
	
	Now let $A$ and $B$ be two algebras and coalgebras (Neither is necessarily a bialgebra). Suppose that $A\#_\sigma B$ is a smash product via $\sigma$ and $A{_\tau\ltimes}B$ is a smash coproduct via $\tau$. If $(A\otimes B,m_{A\#_\sigma B},\eta_{A\#_\sigma B},\varDelta_{A{_\tau\ltimes}B}$, $\epsilon_{A{_\tau\ltimes}B})$ is a bialgebra, then we call this bialgebra a smash biproduct of $A$ and $B$, denoted by $A{_\tau\times_\sigma}B$ (cf. \cite[Definition 4.1]{C&I&M&Z}).
	
	Next we recall the definitions of normal and conormal conditions as in \cite{C&I&M&Z}:
	\begin{defn}
		[\protect{\cite[Definition 2.4]{C&I&M&Z}}] Let $A$ and $B$ be two algebras and $\sigma:B\otimes A\rightarrow A\otimes B$ a linear map. Then $\sigma$ is called left normal if $$\text{(LN)}\quad\sigma\circ(id_B\otimes\eta_A)=\eta_A\otimes id_B;$$
		$\sigma$ is called right normal if $$\text{(RN)}\quad\sigma\circ(\eta_B\otimes id_A)=id_A\otimes\eta_B.$$
	\end{defn}
	\begin{defn}
		[\protect{\cite[Definition 3.3]{C&I&M&Z}}] Let $C$ and $D$ be two coalgebras and $\tau:C\otimes D\rightarrow D\otimes C$ a linear map. Then $\tau$ is called left conormal if  $$\text{(LCN)}\quad(id_D\otimes\epsilon_C)\circ\tau=\epsilon_C\otimes id_D;$$
		$\tau$ is called right conormal if $$\text{(RCN)}\quad(\epsilon_D\otimes id_C)\circ\tau=id_C\otimes\epsilon_D.$$
	\end{defn}
	In \cite{C&I&M&Z}, the normal conditions are used in smash product $A\#_\sigma B$ to relate the multiplication with the units. The conormal conditions are used in $C{_\tau\ltimes}D$ to relate the comultiplication with the counits. Thus for smash biproduct $A{_\tau\times_\sigma}B$, $\sigma$ is always left normal and right normal; $\tau$ is always left conormal and right conormal. 
	
	But in this paper, we consider the following case for $A{_\tau\times_\sigma}B$:
	
	$\bullet$ $\tau$ is left normal: $\tau\circ(id_A\otimes\eta_B)=\eta_B\otimes id_A$, i.e., $\sum 1_{B\tau}\otimes a_{\tau}=1_B\otimes a$;
	
	$\bullet$ $\tau$ is right normal: $\tau\circ(\eta_A\otimes id_B)=id_B\otimes\eta_A$, i.e., $\sum b_{\tau}\otimes1_{A\tau}=b\otimes 1_{A}$;
	
	$\bullet$ $\sigma$ is left conormal: $(id_A\otimes\epsilon_B)\circ\sigma=\epsilon_B\otimes id_A$, i.e., $\sum a_{\sigma}\epsilon_B(b_{\sigma})=a\epsilon_B(b)$;
	
	$\bullet$ $\sigma$ is right conormal: $(\epsilon_A\otimes id_B)\circ\sigma=id_B\otimes\epsilon_A$, i.e., $\sum \epsilon_A(a_{\sigma})b_{\sigma}=\epsilon_A(a)b$.
	
	For general smash biproduct bialgebras, these conditions are not necessarily true. Note that in \cite[Definition and Proposition 2.15]{B&D}, Bespalov and Drabant defined so-called trivalent Hopf data. The four distinct conditions required are exactly these conditions above.
	
	Since $A{_\tau\times_\sigma}B$ is a bialgebra, it's a counital algebra via $\epsilon_{A{_\tau\ltimes}B}=\epsilon_A\otimes\epsilon_B$. But neither $A$ or $B$ is necessarily a counital algebra. This is justified by replacing $(A,m_A,\eta_A,\varDelta_A,\epsilon_A)$ and $(B,m_B,\eta_B,\varDelta_B,\epsilon_B)$ with $(A,m_A,\eta_A,\lambda\varDelta_A,\lambda^{-1}\epsilon_A)$ and $(B,m_B,\eta_B,\lambda^{-1}\varDelta_B,\lambda\epsilon_B)$ for some nonzero $\lambda\in\Bbbk^{\times}$. Similarly though $A{_\tau\times_\sigma}B$ is a unital coalgebra via $\eta_{A\#_\sigma B}=\eta_A\otimes\eta_B$, neither $A$ or $B$ is necessarily a unital coalgebra.
	
	The following proposition is similar to \cite[Theorem 2.15]{M&W}. But we drop the condition that $B$ is a bialgebra.
	
	\begin{prop}\label{thm:2.4}
		Let $A$, $B$ be two counital algebras and unital coalgebras (Neither is necessarily a bialgebra). Suppose that there exist two linear maps $\sigma :B\otimes A\rightarrow A\otimes B$ and $\tau :A\otimes B\rightarrow B\otimes A$. Then $A{_\tau\times_\sigma}B$ is a smash biproduct if and only if the following conditions hold:
		\begin{align*}
			&(B1)\quad A{\#_\sigma} B\text{ is a smash product};\\
			&(B2)\quad A{_\tau\ltimes} B\text{ is a smash coproduct};\\
			&(B3)\quad \tau\circ(\eta_A\otimes \eta_B)=\eta_B\otimes\eta_A,\quad (\epsilon_A\otimes \epsilon_B)\circ \sigma=\epsilon_B\otimes \epsilon_A;\\
			&(B4)\quad \sum (aa')_{(1)}\otimes 1_{B\tau}\otimes (aa')_{(2)\tau}=\sum a_{(1)}a'_{(1)\sigma}\otimes 1_{B\tau_1 \sigma}1_{B\tau_2} \otimes a_{(2)\tau_1}a'_{(2)\tau_2};\\
			&(B5)\quad \sum b_\tau \otimes a_\tau =\sum 1_{B\tau_1}b_{\tau_2}\otimes a_{\tau_1}1_{A\tau_2};\\
			&(B6)\quad \sum (bb')_{(1)\tau}\otimes 1_{A\tau}\otimes (bb')_{(2)}=\sum b_{(1)\tau_1}b'_{(1)\tau_2}\otimes 1_{A\tau_1}1_{A\tau_2\sigma} \otimes b_{(2)\sigma}b'_{(2)};\\
			&(B7)\quad \sum a_{\sigma (1)}\otimes b_{\sigma (1)\tau}\otimes a_{\sigma (2)\tau}\otimes b_{\sigma (2)}=\sum a_{(1)\sigma_1}\otimes b_{(1)\tau_1 \sigma_1}1_{B\tau_2}\otimes 1_{A\tau_1}a_{(2)\tau_2\sigma_2}\\
			&\hspace{4em}\otimes b_{(2)\sigma_2}.
		\end{align*}
	\end{prop}
	Here $(B4)$ means that $\varDelta$ is multiplicative on $(a\otimes 1_B)(a'\otimes 1_B)$, similarly $(B5)$ for $(a\otimes1_B)(1_A\otimes b)$, $(B6)$ for $(1_A\otimes b)(1_A\otimes b')$ and $(B7)$ for $(1_A\otimes b)(a\otimes1_B)$.
	
	Throughout the rest of the paper, when speaking of a smash biproduct bialgebra $A{_\tau\times_\sigma}B$, we always assume that both $A$ and $B$ are counital algebras via $\epsilon_A$, $\epsilon_B$ and unital coalgebras via $\eta_A$, $\eta_B$.

	\section{Factorization of quasitriangular structures of smash biproduct bialgebras}\label{Sec:3}
	In this section, we establish our factorization theory of quasitriangular structures of the smash biproduct bialgebra $A{_\tau\times_\sigma}B$ for which $\sigma$ is right conormal.
	
	We recall the definition of the factor elements of a quasitriangular structure as in \cite{Z&W&J}. Let $A{_\tau\times_\sigma}B$ be a smash biproduct bialgebra. Suppose that $\mathcal{R}=\sum \mathcal{R}^1\otimes \mathcal{R}^2\otimes \mathcal{R}^3\otimes \mathcal{R}^4\in (A{_\tau\times_\sigma}B)\otimes(A{_\tau\times_\sigma}B)$ is a quasitriangular structure of $A{_\tau\times_\sigma}B$. We have the following elements:
	\begin{align*}
		&W=\sum W^1\otimes W^2=(\pi_B\otimes \pi_B)(\mathcal{R})\in B\otimes B,\\
		&X=\sum X^1\otimes X^2=(\pi_A\otimes \pi_B)(\mathcal{R})\in A\otimes B,\\
		&Y=\sum Y^1\otimes Y^2=(\pi_B\otimes \pi_A)(\mathcal{R})\in B\otimes A,\\
		&Z=\sum Z^1\otimes Z^2=(\pi_A\otimes \pi_A)(\mathcal{R})\in A\otimes A,
	\end{align*}
	where $\pi_A=id_A\otimes\epsilon_B:A{_\tau\times_\sigma}B\rightarrow A$ and $\pi_B=\epsilon_A\otimes id_B:A{_\tau\times_\sigma}B\rightarrow B$.
	
	We call these $W$, $X$, $Y$ and $Z$ the factor elements of $\mathcal{R}$. Immediately we get the first properties for these factor elements.
	\begin{prop}\label{thm:3.1}
		Let $A{_\tau\times_\sigma}B$ be a smash biproduct bialgebra. Suppose that $\mathcal{R}=\sum \mathcal{R}^1\otimes \mathcal{R}^2\otimes \mathcal{R}^3\otimes \mathcal{R}^4\in (A{_\tau\times_\sigma}B)\otimes(A{_\tau\times_\sigma}B)$ is a quasitriangular structure of $A{_\tau\times_\sigma}B$. Then
		\begin{align*}
			(N1)\quad&\sum\epsilon_B(W^1)W^2=1_B=\sum W^1\epsilon_B(W^2);\hspace{16em}\\
			(N2)\quad&\sum\epsilon_A(X^1)X^2=1_B,\quad\sum X^1\epsilon_B(X^2)=1_A;\\
			(N3)\quad&\sum\epsilon_B(Y^1)Y^2=1_A,\quad\sum Y^1\epsilon_A(Y^2)=1_B;\\
			(N4)\quad&\sum\epsilon_A(Z^1)Z^2=1_A=\sum Z^1\epsilon_A(Z^2).
		\end{align*}
	\end{prop}
	
	\begin{proof}
		By using $(QT1)$ for $A{_\tau\times_\sigma}B$.
	\end{proof}
	
	For any element $W\in B\otimes B$ satisfying $(N1)$ above, we will call it normalized. It's the same to $X\in A\otimes B$, $Y\in B\otimes A$ and $Z\in A\otimes A$ for $(N2)-(N4)$.
	
	Suppose that $\sigma$ is right conormal. We have the following identities derived from $(B4)$, $(B5)$, $(B6)$ and $(B7)$ of Proposition \ref{thm:2.4}:
	
	\begin{prop}
		Let $A{_\tau\times_\sigma}B$ be a smash biproduct bialgebra for which $\sigma$ is right conormal. Then for $a,a'\in A$, $b,b'\in B$ we have
		\begin{align*}
			(B4.1)\quad &\sum 1_{B\tau}\otimes (aa')_{\tau}=\sum 1_{B\tau_1}1_{B\tau_2}\otimes a_{\tau_1}a'_{\tau_2};\\
			(B6.1)\quad &\sum (bb')_{(1)}\otimes (bb')_{(2)}=\sum b_{(1)}b'_{(1)}\otimes b_{(2)}b'_{(2)};\\
			(B7.1)\quad &\sum a_{\sigma}\otimes b_{\sigma}=\sum a_\sigma\otimes \epsilon_B(b_{(1)\sigma})b_{(2)};\\
			(B8.1)\quad &\sum b_\tau\otimes (aa')_\tau=\sum 1_{B\tau_1}b_{\tau_2}\otimes a_{\tau_1}a'_{\tau_2}.
		\end{align*}
	\end{prop}
	
	\begin{proof}
		$(B4.1)$: apply $\epsilon_A\otimes id_B\otimes id_A$ to $(B4)$;
		
		$(B6.1)$: apply $id_B\otimes \epsilon_A\otimes id_B$ to $(B6)$;
		
		$(B7.1)$: apply $id_A\otimes\epsilon_B\otimes\epsilon_A\otimes id_B$ to $(B7)$;
		
		$(B8.1):\sum b_\tau\otimes (aa')_\tau\stackrel{(B5)}{=}\sum 1_{B\tau_1}b_{\tau_2}\otimes (aa')_{\tau_1}1_{A\tau_2}\stackrel{(B4.1)}{=}\sum 1_{B\tau_1}1_{B\tau_2}b_{\tau_3}\otimes a_{\tau_1}a'_{\tau_2}1_{A\tau_3}$
		$\stackrel{(B5)}{=}1_{B\tau_1}b_{\tau_2}\otimes a_{\tau_1}a'_{\tau_2}$.
	\end{proof}
	
	\begin{rmk}\label{Rmk:3.3}
		$(B6.1)$ indicates that $B$ is a bialgebra. Now $\pi_B:A{_\tau\times_\sigma}B\rightarrow B$ $(\sum_i a_i\otimes b_i\mapsto \sum_i\epsilon_A(a_i)b_i)$ is an algebra and coalgebra surjection. Thus $B$ is a quotient-bialgebra of $A{_\tau\times_\sigma}B$.
	\end{rmk}
	
	For a given quasitriangular structure $\mathcal{R}$ of $A{_\tau\times_\sigma}B$, we can express it by using its factor elements if $\sigma$ is right conormal.
	
	\begin{prop}\label{thm:3.4}
		Let $A{_\tau\times_\sigma}B$ be a smash biproduct bialgebra for which $\sigma$ is right conormal. Suppose that $\mathcal{R}=\sum \mathcal{R}^1\otimes \mathcal{R}^2\otimes \mathcal{R}^3\otimes \mathcal{R}^4\in (A{_\tau\times_\sigma}B)\otimes(A{_\tau\times_\sigma}B)$ is a quasitriangular structure of $A{_\tau\times_\sigma}B$. Then we have
		\begin{align*}
			\mathcal{R}=&\sum \tensor{Z}{^1_{\tau_1\tau_2}}\tensor{\bar{X}}{^1_{\tau_3}}X^1\otimes W^1Y^1\otimes Z^2\tensor{Y}{^2_{\sigma_1\sigma_2}}\epsilon_B(1_{B\tau_1\sigma_2}\tensor{\bar{X}}{^2_{\sigma_1}})\otimes1_{B\tau_2}1_{B\tau_3}X^2W^2.
		\end{align*}
	\end{prop}
	
	Here $\bar{X}=\sum\tensor{\bar{X}}{^1}\otimes\tensor{\bar{X}}{^2}$ is another copy of $X$. The variations of $X$ on the head always mean another copy of $X$, and the same to $W$, $Y$ and $Z$.
	
	\begin{proof}
		By $(QT4)$, for $\forall a\in A$, $b\in B$ we have
		\begin{align}
			&\sum a_{(2)\tau}{\mathcal{R}^1}_{{\sigma}_1} \otimes b_{(2){{\sigma}_1}}{\mathcal{R}^2} \otimes a_{(1)}{\mathcal{R}^3}_{{\sigma}_2} \otimes b_{(1)\tau {{\sigma}_2}}{\mathcal{R}^4}\\
			=&\sum{\mathcal{R}^1}a_{(1){{\sigma}_1}}\otimes {\mathcal{R}^2}_{{\sigma}_1}b_{(1)\tau}\otimes {\mathcal{R}^3}a_{(2)\tau{{\sigma}_2}}\otimes {\mathcal{R}^4}_{{\sigma}_2}b_{(2)}\nonumber.
		\end{align}
		Set $b=1_B$, we have
		\begin{align*}
			\sum&a_{(2)\tau}{\mathcal{R}^1} \otimes {\mathcal{R}^2} \otimes a_{(1)}{\mathcal{R}^3}_{\sigma} \otimes 1_{B\tau {\sigma}}{\mathcal{R}^4}\\
			=&\sum{\mathcal{R}^1}a_{(1){{\sigma}_1}}\otimes {\mathcal{R}^2}_{{\sigma}_1}1_{B\tau}\otimes {\mathcal{R}^3}a_{(2)\tau{{\sigma}_2}}\otimes {\mathcal{R}^4}_{{\sigma}_2}.
		\end{align*}
		Apply $id_A\otimes \epsilon_B\otimes\epsilon_A\otimes id_B$ on both sides, we have
		\begin{align}
			\sum a_{\tau}{\mathcal{R}^1}\epsilon_B({\mathcal{R}^2})\otimes \epsilon_A(\mathcal{R}^3) 1_{B\tau}{\mathcal{R}^4}=\sum{\mathcal{R}^1}a_{\sigma}\epsilon_B( {\mathcal{R}^2}_{\sigma})\otimes \epsilon_A({\mathcal{R}^3}){\mathcal{R}^4}.
		\end{align}
		By $(QT2)$ and $(QT3)$, we calculate $(\varDelta \otimes \varDelta)(\mathcal{R})$, and get
		\begin{align}
			\sum&{\mathcal{R}^1}_{(1)}\otimes {\mathcal{R}^2}_{(1){\tau_1}}\otimes {\mathcal{R}^1}_{(2){\tau_1}}\otimes {\mathcal{R}^2}_{(2)}\otimes {\mathcal{R}^3}_{(1)}\otimes{\mathcal{R}^4}_{(1){\tau_2}}\otimes{\mathcal{R}^3}_{(2){\tau_2}}\\
			&\otimes {\mathcal{R}^4}_{(2)}\nonumber\\
			=\sum&{\mathcal{R}^1}\tensor{\mathcal{\tilde{R}}}{^1_{\sigma_1}}\otimes {\mathcal{R}^2}_{\sigma_1}{\mathcal{\tilde{R}}^2}\otimes {\mathcal{\bar{R}}^1}{\mathcal{\bar{\bar{R}}}^1}_{\sigma_2}\otimes \tensor{\mathcal{\bar{R}}}{^2_{\sigma_2}}{\mathcal{\bar{\bar{R}}}^2}\otimes {\mathcal{\tilde{R}}^3}{\mathcal{\bar{\bar{R}}}^3}_{\sigma_3}\otimes \tensor{\mathcal{\tilde{R}}}{^4_{\sigma_3}}{\mathcal{\bar{\bar{R}}}^4}\nonumber\\
			&\otimes {\mathcal{R}^3}\tensor{\mathcal{\bar{R}}}{^3_{\sigma_4}}\otimes {\mathcal{R}^4}_{\sigma_4}{\mathcal{\bar{R}}^4}\nonumber.
		\end{align}
		Apply $\epsilon_A\otimes \epsilon_B\otimes id_A\otimes \epsilon_B\otimes id_A\otimes \epsilon_B\otimes \epsilon_A\otimes id_B$ to $(3)$, we have
		\begin{align}
			\sum&{\mathcal{R}^1}\epsilon_B(\mathcal{R}^2)\otimes \mathcal{R}^3\otimes\mathcal{R}^4\\
			=\sum&\epsilon_A({\mathcal{R}^1}\tensor{\mathcal{\tilde{R}}}{^1_{\sigma_1}})\epsilon_B( {\mathcal{R}^2}_{\sigma_1}{\mathcal{\tilde{R}}^2}) {\mathcal{\bar{R}}^1}{\mathcal{\bar{\bar{R}}}^1}_{\sigma_2} \epsilon_B(\tensor{\mathcal{\bar{R}}}{^2_{\sigma_2}}{\mathcal{\bar{\bar{R}}}^2})\otimes {\mathcal{\tilde{R}}^3}{\mathcal{\bar{\bar{R}}}^3}_{\sigma_3} \epsilon_B(\tensor{\mathcal{\tilde{R}}}{^4_{\sigma_3}}{\mathcal{\bar{\bar{R}}}^4})\nonumber\\
			&\otimes\epsilon_A({\mathcal{R}^3}\tensor{\mathcal{\bar{R}}}{^3_{\sigma_4}}) {\mathcal{R}^4}_{\sigma_4}{\mathcal{\bar{R}}^4}\nonumber\\
			=\sum&\epsilon_A({\mathcal{R}^1}\tensor{\mathcal{\tilde{R}}}{^1})\epsilon_B( {\mathcal{R}^2}{\mathcal{\tilde{R}}^2}) {\mathcal{\bar{R}}^1}{\mathcal{\bar{\bar{R}}}^1}_{\sigma_2} \epsilon_B(\tensor{\mathcal{\bar{R}}}{^2_{\sigma_2}}{\mathcal{\bar{\bar{R}}}^2})\otimes {\mathcal{\tilde{R}}^3}{\mathcal{\bar{\bar{R}}}^3}_{\sigma_3} \epsilon_B(\tensor{\mathcal{\tilde{R}}}{^4_{\sigma_3}}{\mathcal{\bar{\bar{R}}}^4})\nonumber\\
			&\otimes \epsilon_A({\mathcal{R}^3}\tensor{\mathcal{\bar{R}}}{^3}) {\mathcal{R}^4}{\mathcal{\bar{R}}^4}\nonumber\\
			=\sum&{\mathcal{\bar{R}}^1}{\mathcal{\bar{\bar{R}}}^1}_{\sigma_2} \epsilon_B(\tensor{\mathcal{\bar{R}}}{^2_{\sigma_2}})\epsilon_B({\mathcal{\bar{\bar{R}}}^2})\otimes{\mathcal{\bar{\bar{R}}}^3}\epsilon_B({\mathcal{\bar{\bar{R}}}^4})\otimes \epsilon_A(\tensor{\mathcal{\bar{R}}}{^3}) {\mathcal{R}^4}\nonumber\\
			\stackrel{(2)}{=}\sum&{\mathcal{\bar{\bar{R}}}^1}_{\tau}{\mathcal{\bar{R}}^1} \epsilon_B(\tensor{\mathcal{\bar{R}}}{^2})\epsilon_B({\mathcal{\bar{\bar{R}}}^2})\otimes {\mathcal{\bar{\bar{R}}}^3}\epsilon_B({\mathcal{\bar{\bar{R}}}^4})\otimes1_{B\tau}\epsilon_A(\tensor{\mathcal{\bar{R}}}{^3}) {\mathcal{R}^4}\nonumber\\
			\stackrel{}{=}\sum& {\mathcal{\bar{\bar{R}}}^1}_{\tau}\epsilon_B({\mathcal{\bar{\bar{R}}}^2}){\mathcal{\bar{R}}^1} \epsilon_B(\tensor{\mathcal{\bar{R}}}{^2})\otimes {\mathcal{\bar{\bar{R}}}^3}\epsilon_B({\mathcal{\bar{\bar{R}}}^4})\otimes1_{B\tau}\epsilon_A(\tensor{\mathcal{\bar{R}}}{^3}) {\mathcal{R}^4}\nonumber\\
			=\sum&\tensor{Z}{^1_\tau}X^1\otimes Z^2\otimes 1_{B\tau}X^2\nonumber.
		\end{align}
		Apply $id_A\otimes \epsilon_B\otimes \epsilon_A\otimes id_B\otimes id_A\otimes \epsilon_B\otimes \epsilon_A\otimes id_B$ on $(3)$, we have
		\begin{align*}
			\mathcal{R}=&\sum \mathcal{R}^1\otimes \mathcal{R}^2\otimes \mathcal{R}^3\otimes \mathcal{R}^4\\
			=&\sum {\mathcal{R}^1}\tensor{\mathcal{\tilde{R}}}{^1_{\sigma_1}}\epsilon_B( {\mathcal{R}^2}_{\sigma_1})\epsilon_B({\mathcal{\tilde{R}}^2})\otimes\epsilon_A({\mathcal{\bar{R}}^1})\epsilon_A({\mathcal{\bar{\bar{R}}}^1}_{\sigma_2})\tensor{\mathcal{\bar{R}}}{^2_{\sigma_2}}{\mathcal{\bar{\bar{R}}}^2}\otimes {\mathcal{\tilde{R}}^3}{\mathcal{\bar{\bar{R}}}^3}_{\sigma_3}\epsilon_B( \tensor{\mathcal{\tilde{R}}}{^4_{\sigma_3}})\\
			&\cdot\epsilon_B({\mathcal{\bar{\bar{R}}}^4})\otimes\epsilon_A( {\mathcal{R}^3})\epsilon_A(\tensor{\mathcal{\bar{R}}}{^3_{\sigma_4}}){\mathcal{R}^4}_{\sigma_4}{\mathcal{\bar{R}}^4}\\
			=&\sum {\mathcal{R}^1}\tensor{\mathcal{\tilde{R}}}{^1_{\sigma_1}}\epsilon_B( {\mathcal{R}^2}_{\sigma_1})\epsilon_B({\mathcal{\tilde{R}}^2})\otimes\epsilon_A({\mathcal{\bar{R}}^1})\epsilon_A({\mathcal{\bar{\bar{R}}}^1})\tensor{\mathcal{\bar{R}}}{^2}{\mathcal{\bar{\bar{R}}}^2}\otimes {\mathcal{\tilde{R}}^3}{\mathcal{\bar{\bar{R}}}^3}_{\sigma_3}\epsilon_B( \tensor{\mathcal{\tilde{R}}}{^4_{\sigma_3}})\\
			&\cdot\epsilon_B({\mathcal{\bar{\bar{R}}}^4})\otimes\epsilon_A( {\mathcal{R}^3})\epsilon_A(\tensor{\mathcal{\bar{R}}}{^3}){\mathcal{R}^4}{\mathcal{\bar{R}}^4}\quad(\text{since $\sigma$ is right conormal})\\
			=&\sum {\mathcal{R}^1}\tensor{\mathcal{\tilde{R}}}{^1_{\sigma_1}}\epsilon_B( {\mathcal{R}^2}_{\sigma_1})\epsilon_B({\mathcal{\tilde{R}}^2})\otimes W^1Y^1\otimes {\mathcal{\tilde{R}}^3}\epsilon_B( \tensor{\mathcal{\tilde{R}}}{^4_{\sigma_3}})\tensor{Y}{^2_{\sigma_3}}\otimes\epsilon_A(\tensor{\mathcal{R}}{^3}){\mathcal{R}^4}W^2\\
			\stackrel{(2)}{=}&\sum \tensor{\mathcal{\tilde{R}}}{^1_{\tau}}{\mathcal{R}^1}\epsilon_B( \mathcal{R}^2)\epsilon_B({\mathcal{\tilde{R}}^2})\otimes W^1Y^1\otimes {\mathcal{\tilde{R}}^3}\epsilon_B( \tensor{\mathcal{\tilde{R}}}{^4_{\sigma}})\tensor{Y}{^2_{\sigma}}\otimes1_{B\tau}\epsilon_A(\tensor{\mathcal{R}}{^3}){\mathcal{R}^4}W^2\\
			\stackrel{}{=}&\sum \tensor{\mathcal{\tilde{R}}}{^1_{\tau}}\epsilon_B({\mathcal{\tilde{R}}^2})X^1\otimes W^1Y^1\otimes {\mathcal{\tilde{R}}^3}\epsilon_B( \tensor{\mathcal{\tilde{R}}}{^4_{\sigma}})\tensor{Y}{^2_{\sigma}}\otimes1_{B\tau}X^2W^2\\
			\stackrel{(4)}{=}&\sum (\tensor{Z}{^1_{\tau_1}}{\bar{X}}^1)_{\tau_2}X^1\otimes W^1Y^1\otimes Z^2\epsilon_B((1_{B\tau_1}{\bar{X}}^2)_{\sigma})\tensor{Y}{^2_{\sigma}}\otimes1_{B\tau_2}X^2W^2\\
			\stackrel{}{=}&\sum \tensor{Z}{^1_{\tau_1\tau_2}}\tensor{\bar{X}}{^1_{\tau_3}}X^1\otimes W^1Y^1\otimes Z^2\tensor{Y}{^2_{\sigma_1\sigma_2}}\epsilon_B(1_{B\tau_1\sigma_2}\tensor{\bar{X}}{^2_{\sigma_1}})\otimes1_{B\tau_2}1_{B\tau_3}X^2W^2.
		\end{align*}
	\end{proof}

	\begin{prop}\label{thm:3.5}
		Let $A{_\tau\times_\sigma}B$ be a smash biproduct bialgebra for which $\sigma$ is right conormal. Suppose that there exist normalized elements $W=\sum W^1\otimes W^2\in B\otimes B$, $X=\sum X^1\otimes X^2\in A\otimes B$, $Y=\sum Y^1\otimes Y^2\in B\otimes A$ and $Z=\sum Z^1\otimes Z^2\in A\otimes A$ such that $\alpha=\sum \tensor{Z}{^1_{\tau_1\tau_2}}\tensor{\bar{X}}{^1_{\tau_3}}X^1\otimes W^1Y^1\otimes Z^2\tensor{Y}{^2_{\sigma_1\sigma_2}}\epsilon_B(1_{B\tau_1\sigma_2}\tensor{\bar{X}}{^2_{\sigma_1}})\otimes1_{B\tau_2}1_{B\tau_3}X^2W^2$ is a quasitriangular structure of $A{_\tau\times_\sigma}B$. Then we have the following identities:
		\begin{align*}
			C1)\quad\sum& \tensor{Z}{^1_{\tau_1\tau_2}}\tensor{\bar{X}}{^1_{\tau_3}}X^1a_{(1)\sigma_3}\epsilon_B( (W^1Y^1)_{\sigma_3})\otimes Z^2\tensor{Y}{^2_{\sigma_1\sigma_2}}a_{(2)\sigma_4}\epsilon_B(1_{B\tau_1\sigma_2}\tensor{\bar{X}}{^2_{\sigma_1}})\\
			&\cdot\epsilon_B((1_{B\tau_2}1_{B\tau_3}X^2W^2)_{\sigma_4})=\sum a_{(2)\tau}\tensor{Z}{^1}\otimes a_{(1)}\tensor{Z}{^2_\sigma}\epsilon_B(1_{B\tau\sigma});\\
			C2)\quad\sum&X^1a_\sigma \epsilon_B(\tensor{W}{^1_\sigma})\otimes X^2W^2=\sum a_\tau X^1\otimes1_{B\tau}\tensor{X}{^2};\\
			C3)\quad\sum&W^1Y^11_{B\tau}\otimes Y^2a_{\tau\sigma}\epsilon_B(\tensor{W}{^2_\sigma})=\sum Y^1\otimes aY^2;\\
			C4)\quad\sum&Z^1\otimes Z^2\epsilon_B(b)=\sum1_{A\tau}\tensor{Z}{^1_{\sigma_1}}\otimes\tensor{Z}{^2_{\sigma_2}}\epsilon_B(b_{(1)\tau\sigma_2})\epsilon_B(b_{(2)\sigma_1});\\
			C5)\quad\sum&X^1\otimes X^2b=\sum1_{A\tau}\tensor{X}{^1_{\sigma}}\epsilon_B(b_{(2)\sigma})\otimes b_{(1)\tau}X^2;\\
			C6)\quad\sum&W^1Y^1b_\tau\otimes Y^21_{A\tau\sigma}\epsilon_B(\tensor{W}{^2_\sigma})=\sum b_{\sigma} Y^1\otimes \tensor{Y}{^2_\sigma};\\
			C7)\quad\sum&W^1b_{(1)}\otimes W^2b_{(2)}=\sum b_{(2)}W^1\otimes b_{(1)}W^2;\\
			C8)\quad\sum&\tensor{Z}{^1_{(1)}}\otimes \tensor{Z}{^1_{(2)}}\otimes Z^2=\sum \tensor{Z}{^1_\tau}X^1\otimes\bar{Z}^1\otimes Z^2\tensor{\bar{Z}}{^2_\sigma}\epsilon_B((1_{B\tau}\tensor{X}{^2})_{\sigma});\\
			C9)\quad\sum&\tensor{X}{^1_{(1)}}\otimes \tensor{X}{^1_{(2)}}\otimes X^2=\sum X^1\otimes{\bar{X}}^1\otimes X^2{\bar{X}}^2;\\
			C10)\quad\sum&\tensor{Y}{^1_{\tau_1}}\otimes (\tensor{Z}{^1_{\tau_2}}X^1)_{\tau_1}\otimes Z^2\tensor{Y}{^2_\sigma}\epsilon_B((1_{B\tau_2}X^2)_{\sigma})=\sum W^1Y^1\otimes Z^1\\
			&\otimes Y^2\tensor{Z}{^2_\sigma}\epsilon_B(\tensor{W}{^2_\sigma});\\
			C11)\quad\sum&\tensor{W}{^1_\tau}\otimes \tensor{X}{^1_\tau}\otimes X^2W^2=\sum W^1\otimes X^1\otimes W^2X^2;\\
			C12)\quad\sum&\tensor{Y}{^1_{(1)}}\otimes \tensor{Y}{^1_{(2)}}\otimes Y^2=\sum W^1Y^1\otimes {\bar{Y}}^1\otimes Y^2\tensor{\bar{Y}}{^2_\sigma}\epsilon_B(\tensor{W}{^2_\sigma});\\
			C13)\quad\sum&\tensor{W}{^1_{(1)}}\otimes \tensor{W}{^1_{(2)}}\otimes W^2=\sum W^1\otimes{\bar{W}}^1\otimes W^2{\bar{W}}^2;\\
			C14)\quad\sum&Z^1\otimes\tensor{Z}{^2_{(1)}}\otimes \tensor{Z}{^2_{(2)}}=\sum \tensor{Z}{^1_\tau}X^1\tensor{\bar{Z}}{^1_{\sigma_1}}\epsilon_B(\tensor{Y}{^1_{\sigma_1}})\otimes{\bar{Z}}^2\otimes Z^2\tensor{Y}{^2_{\sigma_2}}\\
			&\cdot\epsilon_B((1_{B\tau}\tensor{X}{^2})_{\sigma_2});\\
			C15)\quad\sum&\tensor{Z}{^1_{\tau_1}} X^1\otimes (1_{B\tau_1}\tensor{X}{^2})_{\tau_2}\otimes\tensor{Z}{^2_{\tau_2}}=\sum \tensor{Z}{^1_\tau}X^1\tensor{\bar{X}}{^1_{\sigma_1}}\epsilon_B(\tensor{Y}{^1_{\sigma_1}})\otimes {\bar{X}}^2\\
			&\otimes Z^2\tensor{Y}{^2_{\sigma_2}}\epsilon_B((1_{B\tau}\tensor{X}{^2})_{\sigma_2});\\
			C16)\quad\sum&X^1\otimes\tensor{X}{^2_{(1)}}\otimes \tensor{X}{^2_{(2)}}=\sum X^1\tensor{\bar{X}}{^1_\sigma}\epsilon_B(\tensor{W}{^1_\sigma}) \otimes{\bar{X}}^2\otimes X^2W^2;\\
			C17)\quad\sum&Y^1\otimes\tensor{Y}{^2_{(1)}}\otimes \tensor{Y}{^2_{(2)}}=\sum Y^1{\bar{Y}}^1\otimes {\bar{Y}}^2\otimes Y^2;\\
			C18)\quad\sum&W^1Y^1\otimes\tensor{W}{^2_\tau}\otimes \tensor{Y}{^2_\tau}=\sum Y^1W^1 \otimes W^2\otimes Y^2;\\
			C19)\quad\sum&W^1\otimes\tensor{W}{^2_{(1)}}\otimes \tensor{W}{^2_{(2)}}=\sum W^1{\bar{W}}^1 \otimes {\bar{W}}^2\otimes W^2.
		\end{align*}
	\end{prop}
	\begin{proof}
		By $(QT4)$ for $\alpha$, for $\forall a\in A$, $b\in B$ we have
		\begin{align}
			\sum&\tensor{Z}{^1_{\tau_1\tau_2}}\tensor{\bar{X}}{^1_{\tau_3}}X^1a_{(1)\sigma_3}\otimes (W^1Y^1)_{\sigma_3}b_{(1)\tau_4}\otimes Z^2\tensor{Y}{^2_{\sigma_1\sigma_2}}a_{(2)\tau_4\sigma_4}\\
			&\cdot\epsilon_B(1_{B\tau_1\sigma_2}\tensor{\bar{X}}{^2_{\sigma_1}})\otimes (1_{B\tau_2}1_{B\tau_3}X^2W^2)_{\sigma_4}b_{(2)}\nonumber\\
			=\sum& a_{(2)\tau_4}(\tensor{Z}{^1_{\tau_1\tau_2}}\tensor{\bar{X}}{^1_{\tau_3}}X^1)_{\sigma_3}\otimes b_{(2)\sigma_3}W^1Y^1\otimes a_{(1)}(Z^2\tensor{Y}{^2_{\sigma_1\sigma_2}})_{\sigma_4}\nonumber\\
			&\cdot\epsilon_B(1_{B\tau_1\sigma_2}\tensor{\bar{X}}{^2_{\sigma_1}})\otimes b_{(1)\tau_4\sigma_4}	1_{B\tau_2}1_{B\tau_3}X^2W^2\nonumber.
		\end{align}
		Set $b=1_B$, then
		\begin{align}
			\sum&\tensor{Z}{^1_{\tau_1\tau_2}}\tensor{\bar{X}}{^1_{\tau_3}}X^1a_{(1)\sigma_3}\otimes (W^1Y^1)_{\sigma_3}1_{B\tau_4}\otimes Z^2\tensor{Y}{^2_{\sigma_1\sigma_2}}a_{(2)\tau_4\sigma_4}\\
			&\cdot\epsilon_B(1_{B\tau_1\sigma_2}\tensor{\bar{X}}{^2_{\sigma_1}})\otimes (1_{B\tau_2}1_{B\tau_3}X^2W^2)_{\sigma_4}\nonumber\\
			=\sum& a_{(2)\tau_4}\tensor{Z}{^1_{\tau_1\tau_2}}\tensor{\bar{X}}{^1_{\tau_3}}X^1\otimes W^1Y^1\otimes a_{(1)}(Z^2\tensor{Y}{^2_{\sigma_1\sigma_2}})_{\sigma_3}\epsilon_B(1_{B\tau_1\sigma_2}\tensor{\bar{X}}{^2_{\sigma_1}})\nonumber\\
			&\otimes 1_{B\tau_4\sigma_3}1_{B\tau_2}1_{B\tau_3}X^2W^2\nonumber.
		\end{align}
		By applying $id_A\otimes\epsilon_B\otimes id_A\otimes\epsilon_B$, $id_A\otimes\epsilon_B\otimes\epsilon_A\otimes id_B$ and $\epsilon_A\otimes id_B\otimes id_A\otimes \epsilon_B$, we get $C1)$, $C2)$ and $C3)$, respectively. Note that by applying $\epsilon_A\otimes id_B\otimes \epsilon_A\otimes id_B$, we get the identity $W\equiv W$.
		
		If we set $a=1_A$ in $(5)$, we get
		\begin{align}
			\sum&\tensor{Z}{^1_{\tau_1\tau_2}}\tensor{\bar{X}}{^1_{\tau_3}}X^1\otimes W^1Y^1b_{(1)\tau_4}\otimes Z^2\tensor{Y}{^2_{\sigma_1\sigma_2}}1_{A\tau_4\sigma_3}\epsilon_B(1_{B\tau_1\sigma_2}\tensor{\bar{X}}{^2_{\sigma_1}})\\
			&\otimes (1_{B\tau_2}1_{B\tau_3}X^2W^2)_{\sigma_3}b_{(2)}\nonumber\\
			=\sum& 1_{A\tau_4}(\tensor{Z}{^1_{\tau_1\tau_2}}\tensor{\bar{X}}{^1_{\tau_3}}X^1)_{\sigma_3}\otimes b_{(2)\sigma_3}W^1Y^1\otimes(Z^2\tensor{Y}{^2_{\sigma_1\sigma_2}})_{\sigma_4}\epsilon_B(1_{B\tau_1\sigma_2}\tensor{\bar{X}}{^2_{\sigma_1}})\nonumber\\
			&\otimes b_{(1)\tau_4\sigma_4}1_{B\tau_2}1_{B\tau_3}X^2W^2\nonumber.
		\end{align}
		Similarly by applying $id_A\otimes\epsilon_B\otimes id_A\otimes \epsilon_B$, $id_A\otimes\epsilon_B\otimes\epsilon_A\otimes id_B$, $\epsilon_A\otimes id_B\otimes id_A\otimes \epsilon_B$ and $\epsilon_A\otimes id_B\otimes \epsilon_A\otimes id_B$, we get $C4)$, $C5)$, $C6)$ and $C7)$, respectively.
		
		By $(QT2)$ for $\alpha$, we have
		\begin{align}
			\sum&(\tensor{Z}{^1_{\tau_1\tau_2}}\tensor{\bar{X}}{^1_{\tau_3}}X^1)_{(1)}\otimes(W^1Y^1)_{(1)\tau_4}\otimes(\tensor{Z}{^1_{\tau_1\tau_2}}\tensor{\bar{X}}{^1_{\tau_3}}X^1)_{(2)\tau_4}\otimes (W^1Y^1)_{(2)}\\
			&\otimes Z^2\tensor{Y}{^2_{\sigma_1\sigma_2}}\epsilon_B(1_{B\tau_1\sigma_2}\tensor{\bar{X}}{^2_{\sigma_1}})\otimes1_{B\tau_2}1_{B\tau_3}X^2W^2\nonumber\\
			=\sum&\tensor{Z}{^1_{\tau_1\tau_2}}\tensor{\bar{X}}{^1_{\tau_3}}\tensor{X}{^1}\otimes W^1Y^1\otimes\tensor{\bar{Z}}{^1_{\tau_4\tau_5}}\tensor{\tilde{X}}{^1_{\tau_6}}\tensor{\hat{X}}{^1}\otimes \bar{W}^1\bar{Y}^1\nonumber\\
			&\otimes Z^2\tensor{Y}{^2_{\sigma_1\sigma_2}}(\tensor{\bar{Z}}{^2}\tensor{\bar{Y}}{^2_{\sigma_3\sigma_4}})_{\sigma_5}\epsilon_B(1_{B\tau_1\sigma_2}\tensor{\bar{X}}{^2_{\sigma_1}})\epsilon_B(1_{B\tau_4\sigma_4}\tensor{\tilde{X}}{^2_{\sigma_3}})\nonumber\\
			&\otimes(1_{B\tau_2}1_{B\tau_3}X^2W^2)_{\sigma_5}1_{B\tau_5}1_{B\tau_6}\tensor{\hat{X}}{^2}\tensor{\bar{W}}{^2}\nonumber.
		\end{align}
		By applying $id_A\otimes\epsilon_B\otimes id_A\otimes\epsilon_B\otimes id_A\otimes\epsilon_B$, $id_A\otimes\epsilon_B\otimes id_A\otimes\epsilon_B\otimes\epsilon_A\otimes id_B$, $\epsilon_A\otimes id_B\otimes id_A\otimes\epsilon_B\otimes id_A\otimes\epsilon_B$, $\epsilon_A\otimes id_B\otimes id_A\otimes\epsilon_B\otimes\epsilon_A\otimes id_B$, $\epsilon_A\otimes id_B\otimes\epsilon_A\otimes id_B\otimes id_A\otimes\epsilon_B$ and $\epsilon_A\otimes id_B\otimes\epsilon_A\otimes id_B\otimes\epsilon_A\otimes id_B$, we get $C8)$, $C9)$, $C10)$, $C11)$, $C12)$ and $C13)$, respectively. Note that  $id_A\otimes\epsilon_B\otimes\epsilon_A\otimes id_B\otimes id_A\otimes\epsilon_B$ gives the identity $\sum \tensor{Z}{^1_\tau}{\bar{X}}^1\otimes Y^1\otimes Z^2\tensor{Y}{^2_{\sigma_1\sigma_2}}\epsilon_B(1_{B\tau\sigma_2}\tensor{\bar{X}}{^2_{\sigma_1}})=\sum \tensor{Z}{^1_\tau}X^1\otimes {\bar{Y}}^1\otimes Z^2\tensor{\bar{Y}}{^2_\sigma}\epsilon_B((1_{B\tau}X^2)_{\sigma})$ and $id_A\otimes\epsilon_B\otimes\epsilon_A\otimes id_B\otimes\epsilon_A\otimes id_B$ gives the identity $\sum X^1\otimes W^1\otimes X^2W^2=\sum X^1\otimes {\bar{W}}^1\otimes X^2{\bar{W}}^2$.

		By $(QT3)$ for $\alpha$, we have
		\begin{align}
			\sum&\tensor{Z}{^1_{\tau_1\tau_2}}\tensor{\bar{X}}{^1_{\tau_3}}X^1\otimes W^1Y^1\otimes (Z^2\tensor{Y}{^2_{\sigma_1\sigma_2}})_{(1)}\epsilon_B(1_{B\tau_1\sigma_2}\tensor{\bar{X}}{^2_{\sigma_1}})\\
			&\otimes(1_{B\tau_2}1_{B\tau_3}X^2W^2)_{(1){\tau_4}}\otimes(Z^2\tensor{Y}{^2_{\sigma_1\sigma_2}})_{(2){\tau_4}}\otimes(1_{B\tau_2}1_{B\tau_3}X^2W^2)_{(2)}\nonumber\\
			=\sum&\tensor{Z}{^1_{\tau_1\tau_2}}\tensor{\bar{X}}{^1_{\tau_3}}X^1(\tensor{\bar{Z}}{^1_{\tau_4\tau_5}}\tensor{\tilde{X}}{^1_{\tau_6}}{\hat{X}}^1)_{\sigma_5}\otimes (W^1Y^1)_{\sigma_5}\bar{W}^1\bar{Y}^1\nonumber\\
			&\otimes {\bar{Z}}^2\tensor{\bar{Y}}{^2_{\sigma_3\sigma_4}}\epsilon_B(1_{B\tau_4\sigma_4}\tensor{\tilde{X}}{^2_{\sigma_3}})\otimes1_{B\tau_5}1_{B\tau_6}{\hat{X}}^2{\bar{W}}^2\nonumber\\
			&\otimes Z^2\tensor{Y}{^2_{\sigma_1\sigma_2}}\epsilon_B(1_{B\tau_1\sigma_2}\tensor{\bar{X}}{^2_{\sigma_1}})\otimes1_{B\tau_2}1_{B\tau_3}X^2W^2\nonumber.
		\end{align}
		Similarly by applying $id_A\otimes\epsilon_B\otimes id_A\otimes\epsilon_B\otimes id_A\otimes\epsilon_B$, $id_A\otimes\epsilon_B\otimes\epsilon_A\otimes id_B\otimes id_A\otimes\epsilon_B$, $id_A\otimes\epsilon_B\otimes\epsilon_A\otimes id_B\otimes\epsilon_A\otimes id_B$, $\epsilon_A\otimes id_B\otimes id_A\otimes\epsilon_B\otimes id_A\otimes\epsilon_B$, $\epsilon_A\otimes id_B\otimes\epsilon_A\otimes id_B\otimes id_A\otimes\epsilon_B$ and $\epsilon_A\otimes id_B\otimes\epsilon_A\otimes id_B\otimes\epsilon_A\otimes id_B$, we get $C14)$, $C15)$, $C16)$, $C17)$, $C18)$ and $C19)$, respectively. Note that $id_A\otimes\epsilon_B\otimes id_A\otimes\epsilon_B\otimes\epsilon_A\otimes id_B$ gives the identity $\sum \tensor{Z}{^1_\tau} X^1\otimes Z^2\otimes 1_{B\tau}\tensor{X}{^2}=\sum \tensor{\bar{Z}}{^1_\tau} X^1\otimes {\bar{Z}}^2\otimes 1_{B\tau}\tensor{X}{^2}$ and $\epsilon_A\otimes id_B\otimes id_A\otimes\epsilon_B\otimes\epsilon_A\otimes id_B$ gives the identity $\sum W^1Y^1\otimes Y^2\otimes W^2=\sum W^1\bar{Y}^1\otimes \bar{Y}^2\otimes W^2$.
	\end{proof}
	
	We see that $C7),C13)$ and $C19)$ altogether make $W$ a quasitriangular structure of bialgebra $B$. Since $\pi_B:A{_\tau\times_\sigma}B\rightarrow B$ is a bialgebra surjection, it's direct that $W=(\pi_B\otimes\pi_B)(\mathcal{R})$ is a quasitriangular structure of $B$.
	
	Next we prove that we can contrarily reconstruct the quasitriangular structure from the $W$, $X$, $Y$ and $Z$.
	
	\begin{thm}\label{thm:3.6}
		Let $A{_\tau\times_\sigma}B$ be a smash biproduct bialgebra for which $\sigma$ is right conormal. Suppose that there exist normalized elements $W=\sum W^1\otimes W^2\in B\otimes B$, $X=\sum X^1\otimes X^2\in A\otimes B$, $Y=\sum Y^1\otimes Y^2\in B\otimes A$ and $Z=\sum Z^1\otimes Z^2\in A\otimes A$ satisfying $C1)-C19)$ above. Then $\alpha=\sum \tensor{Z}{^1_{\tau_1\tau_2}}\tensor{\bar{X}}{^1_{\tau_3}}X^1\otimes W^1Y^1\otimes Z^2\tensor{Y}{^2_{\sigma_1\sigma_2}}\epsilon_B(1_{B\tau_1\sigma_2}\tensor{\bar{X}}{^2_{\sigma_1}})\otimes1_{B\tau_2}1_{B\tau_3}X^2W^2$ is a quasitriangular structure of $A{_\tau\times_\sigma}B$.
	\end{thm}

	\begin{proof}
		$\bullet$ $(QT1)$ is easy to check.
		
		$\bullet\enspace$The proof of $(QT2)$, $(QT3)$ and $(QT4)$ would be given following the strategy of “\textit{build-upwards}”. Taking $(QT3)$ for example,  by $C16)$ and $C19)$ we could prove an identity which is the result of applying $id_A\otimes id_B\otimes\epsilon_A\otimes id_B\otimes \epsilon_A\otimes id_B$ to equation $(9)$. Similarly another identity corresponding to $id_A\otimes id_B\otimes id_A\otimes\epsilon_B\otimes\epsilon_A\otimes id_B$ would hold naturally. Combining them together, we would get an identity corresponding to $id_A\otimes id_B\otimes id_A\otimes id_B\otimes\epsilon_A\otimes id_B$. Also in another line we get identities corresponding to $id_A\otimes id_B\otimes\epsilon_A\otimes id_B\otimes id_A\otimes\epsilon_B$ by $C15)$ and $C18)$, and corresponding to $id_A\otimes id_B\otimes id_A\otimes\epsilon_B\otimes id_A\otimes\epsilon_B$ by $C14)$ and $C17)$. They lead to an identity corresponding to $id_A\otimes id_B\otimes id_A\otimes id_B\otimes id_A\otimes\epsilon_B$. Finally by combining the identities in two lines, we could prove an identity corresponding to $id_A\otimes id_B\otimes id_A\otimes id_B\otimes id_A\otimes id_B$ which is $(9)$ itself. Thus we get $(QT3)$ for $\alpha$. This method will also work for $(QT2)$ and $(QT4)$.
		
		For instance, we show how to get the identities above.
		
		Identity corr. to $id_A\otimes id_B\otimes\epsilon_A\otimes id_B\otimes\epsilon_A\otimes id_B$, denoted by $(*)$:
		\begin{align*}
			(id_A\otimes&id_B\otimes\epsilon_A\otimes id_B\otimes\epsilon_A\otimes id_B)\Big(\textit{LS of (9)}\Big)\\
			=\sum&\tensor{X}{^1}\otimes \tensor{W}{^1}\otimes (\tensor{X}{^2}\tensor{W}{^2})_{(1)}\otimes(\tensor{X}{^2}\tensor{W}{^2})_{(2)}\\
			\stackrel{(B6.1)}{=}\sum&\tensor{X}{^1}\otimes \tensor{W}{^1}\otimes \tensor{X}{^2_{(1)}}\tensor{W}{^2_{(1)}}\otimes\tensor{X}{^2_{(2)}}\tensor{W}{^2_{(2)}}\\
			\stackrel{\bm{C16)}}{=}\sum&\tensor{X}{^1}\tensor{\bar{X}}{^1_{\sigma}}\epsilon_B(\tensor{\tilde{W}}{^1_{\sigma}})\otimes \tensor{W}{^1}\otimes \tensor{\bar{X}}{^2}\tensor{W}{^2_{(1)}}\otimes\tensor{X}{^2}\tensor{\tilde{W}}{^2}\tensor{W}{^2_{(2)}}\\
			\stackrel{\bm{C19)}}{=}\sum&\tensor{X}{^1}\tensor{\bar{X}}{^1_{\sigma}}\epsilon_B(\tensor{\tilde{W}}{^1_{\sigma}})\otimes \tensor{W}{^1}\tensor{\bar{W}}{^1}\otimes \tensor{\bar{X}}{^2}\tensor{\bar{W}}{^2}\otimes\tensor{X}{^2}\tensor{\tilde{W}}{^2}\tensor{W}{^2}\\
			\stackrel{C13)}{=}\sum&\tensor{X}{^1}\tensor{\bar{X}}{^1_{\sigma}}\epsilon_B(\tensor{W}{^1_{(1)\sigma}})\otimes \tensor{W}{^1_{(2)}}\tensor{\bar{W}}{^1}\otimes \tensor{\bar{X}}{^2}\tensor{\bar{W}}{^2}\otimes\tensor{X}{^2}\tensor{W}{^2}\\
			\stackrel{(B7.1)}{=}\sum&\tensor{X}{^1}\tensor{\bar{X}}{^1_{\sigma}}\otimes \tensor{W}{^1_{\sigma}}\tensor{\bar{W}}{^1}\otimes \tensor{\bar{X}}{^2}\tensor{\bar{W}}{^2}\otimes\tensor{X}{^2}\tensor{W}{^2}\\
			=(id_A&\otimes id_B\otimes\epsilon_A\otimes id_B\otimes\epsilon_A\otimes id_B)\Big(\textit{RS of (9)}\Big).\hspace{16em}
		\end{align*}
		
		Identity corr. to $id_A\otimes id_B\otimes id_A\otimes\epsilon_B\otimes\epsilon_A\otimes id_B$, denoted by $(**)$:
		\begin{align*}
			(id_A\otimes&id_B\otimes id_A\otimes\epsilon_B\otimes\epsilon_A\otimes id_B)\Big(\textit{LS of (9)}\Big)\\
			=\sum&\tensor{Z}{^1_{\tau_1\tau_2}}\tensor{\bar{X}}{^1_{\tau_3}}\tensor{X}{^1}\otimes \tensor{W}{^1}\tensor{Y}{^1}\otimes\tensor{Z}{^2}\tensor{Y}{^2_{\sigma_1\sigma_2}}\epsilon_B(1_{B\tau_1\sigma_2}\tensor{\bar{X}}{^2_{\sigma_1}})\otimes1_{B\tau_2}1_{B\tau_3}\tensor{X}{^2}\tensor{W}{^2}\\
			\stackrel{(B8.1)}{=}\sum&(\tensor{Z}{^1_{\tau_1}}\tensor{\bar{X}}{^1})_{\tau_2}\tensor{X}{^1}\otimes \tensor{W}{^1}\tensor{Y}{^1}\otimes\tensor{Z}{^2}\tensor{Y}{^2_{\sigma_1\sigma_2}}\epsilon_B(1_{B\tau_1\sigma_2}\tensor{\bar{X}}{^2_{\sigma_1}})\otimes1_{B\tau_2}\tensor{X}{^2}\tensor{W}{^2}\\	\stackrel{C2)}{=}\sum&\tensor{X}{^1}(\tensor{Z}{^1_{\tau_1}}\tensor{\bar{X}}{^1})_{\sigma_3}\epsilon_B(\tensor{\bar{W}}{^1_{\sigma_3}})\otimes \tensor{W}{^1}\tensor{Y}{^1}\otimes\tensor{Z}{^2}\tensor{Y}{^2_{\sigma_1\sigma_2}}\epsilon_B(1_{B\tau_1\sigma_2}\tensor{\bar{X}}{^2_{\sigma_1}})\otimes\tensor{X}{^2}\tensor{\bar{W}}{^2}\tensor{W}{^2}\\
			\stackrel{C13)}{=}\sum&\tensor{X}{^1}(\tensor{Z}{^1_{\tau_1}}\tensor{\bar{X}}{^1})_{\sigma_3}\epsilon_B(\tensor{W}{^1_{(1)\sigma_3}})\otimes \tensor{W}{^1_{(2)}}\tensor{Y}{^1}\otimes\tensor{Z}{^2}\tensor{Y}{^2_{\sigma_1\sigma_2}}\epsilon_B(1_{B\tau_1\sigma_2}\tensor{\bar{X}}{^2_{\sigma_1}})\otimes\tensor{X}{^2}\tensor{W}{^2}\\
			\stackrel{(B7.1)}{=}\sum&\tensor{X}{^1}(\tensor{Z}{^1_{\tau_1}}\tensor{\bar{X}}{^1})_{\sigma_3}\otimes \tensor{W}{^1_{\sigma_3}}\tensor{Y}{^1}\otimes\tensor{Z}{^2}\tensor{Y}{^2_{\sigma_1\sigma_2}}\epsilon_B(1_{B\tau_1\sigma_2}\tensor{\bar{X}}{^2_{\sigma_1}})\otimes\tensor{X}{^2}\tensor{W}{^2}\\
			=(id_A&\otimes id_B\otimes id_A\otimes\epsilon_B\otimes\epsilon_A\otimes id_B)\Big(\textit{RS of (9)}\Big).
		\end{align*}
		
		Identity corr. to $id_A\otimes id_B\otimes id_A\otimes id_B\otimes\epsilon_A\otimes id_B$, denoted by $(***)$:
		\begin{align*}
			(id_A\otimes&id_B\otimes id_A\otimes id_B\otimes\epsilon_A\otimes id_B)\Big(\textit{LS of (9)}\Big)\\
			=\sum&\tensor{Z}{^1_{\tau_1\tau_2}}\tensor{\bar{X}}{^1_{\tau_3}}\tensor{X}{^1}\otimes \tensor{W}{^1}\tensor{Y}{^1}\otimes Z^2\tensor{Y}{^2_{\sigma_1\sigma_2}}\epsilon_B((1_{B\tau_1\sigma_2}\tensor{\bar{X}}{^2_{\sigma_1}}))\\
			&\otimes(1_{B\tau_2}1_{B\tau_3}\tensor{X}{^2}\tensor{W}{^2})_{(1)}\otimes(1_{B\tau_2}1_{B\tau_3}\tensor{X}{^2}\tensor{W}{^2})_{(2)}\\
			\stackrel{(B6.1)}{=}\sum&\tensor{Z}{^1_{\tau_1\tau_2}}\tensor{\bar{X}}{^1_{\tau_3}}\tensor{X}{^1}\otimes \tensor{W}{^1}\tensor{Y}{^1}\otimes Z^2\tensor{Y}{^2_{\sigma_1\sigma_2}}\epsilon_B(1_{B\tau_1\sigma_2}\tensor{\bar{X}}{^2_{\sigma_1}})\\
			&\otimes1_{B\tau_2(1)}1_{B\tau_3(1)}(\tensor{X}{^2}\tensor{W}{^2})_{(1)}\otimes1_{B\tau_2(2)}1_{B\tau_3(2)}(\tensor{X}{^2}\tensor{W}{^2})_{(2)}\\
			\stackrel{(1.3)}{=}\sum&\tensor{Z}{^1_{\tau_1\tau_2\tau_4}}\tensor{\bar{X}}{^1_{\tau_3\tau_5}}\tensor{X}{^1}\otimes \tensor{W}{^1}\tensor{Y}{^1}\otimes Z^2\tensor{Y}{^2_{\sigma_1\sigma_2}}\epsilon_B(1_{B\tau_1\sigma_2}\tensor{\bar{X}}{^2_{\sigma_1}})\\
			&\otimes1_{B\tau_2}1_{B\tau_3}(\tensor{X}{^2}\tensor{W}{^2})_{(1)}\otimes1_{B\tau_4}1_{B\tau_5}(\tensor{X}{^2}\tensor{W}{^2})_{(2)}\\
			\stackrel{\bm{(*)}}{=}\sum&\tensor{Z}{^1_{\tau_1\tau_2\tau_4}}\tensor{\bar{X}}{^1_{\tau_3\tau_5}}\tensor{X}{^1}\tensor{\tilde{X}}{^1_{\sigma_3}}\otimes \tensor{W}{^1_{\sigma_3}}\tensor{\tilde{W}}{^1}\tensor{Y}{^1}\otimes Z^2\tensor{Y}{^2_{\sigma_1\sigma_2}}\epsilon_B(1_{B\tau_1\sigma_2}\tensor{\bar{X}}{^2_{\sigma_1}})\\
			&\otimes1_{B\tau_2}1_{B\tau_3}\tensor{\tilde{X}}{^2}\tensor{\tilde{W}}{^2}\otimes1_{B\tau_4}1_{B\tau_5}\tensor{X}{^2}\tensor{W}{^2}\\
			\stackrel{(B7.1)}{=}\sum&\tensor{Z}{^1_{\tau_1\tau_2\tau_4}}\tensor{\bar{X}}{^1_{\tau_3\tau_5}}\tensor{X}{^1}\tensor{\tilde{X}}{^1_{\sigma_3}}\otimes \epsilon_B(\tensor{W}{^1_{(1)\sigma_3}})\tensor{W}{^1_{(2)}}\tensor{\tilde{W}}{^1}\tensor{Y}{^1}\otimes Z^2\tensor{Y}{^2_{\sigma_1\sigma_2}}\\
			&\cdot\epsilon_B(1_{B\tau_1\sigma_2}\tensor{\bar{X}}{^2_{\sigma_1}})\otimes1_{B\tau_2}1_{B\tau_3}\tensor{\tilde{X}}{^2}\tensor{\tilde{W}}{^2}\otimes1_{B\tau_4}1_{B\tau_5}\tensor{X}{^2}\tensor{W}{^2}\\
			\stackrel{C13)}{=}\sum&\tensor{Z}{^1_{\tau_1\tau_2\tau_4}}\tensor{\bar{X}}{^1_{\tau_3\tau_5}}\tensor{X}{^1}\tensor{\tilde{X}}{^1_{\sigma_3}}\otimes \epsilon_B(\tensor{W}{^1_{\sigma_3}})\tensor{\bar{W}}{^1}\tensor{\tilde{W}}{^1}\tensor{Y}{^1}\otimes Z^2\tensor{Y}{^2_{\sigma_1\sigma_2}}\epsilon_B(1_{B\tau_1\sigma_2}\tensor{\bar{X}}{^2_{\sigma_1}})\\
			&\otimes1_{B\tau_2}1_{B\tau_3}\tensor{\tilde{X}}{^2}\tensor{\tilde{W}}{^2}\otimes1_{B\tau_4}1_{B\tau_5}\tensor{X}{^2}\tensor{W}{^2}\tensor{\bar{W}}{^2}\\
			\stackrel{C2)}{=}\sum&\tensor{Z}{^1_{\tau_1\tau_2\tau_4}}\tensor{\bar{X}}{^1_{\tau_3\tau_5}}\tensor{\tilde{X}}{^1_{\tau_6}}\tensor{X}{^1}\otimes \tensor{\bar{W}}{^1}\tensor{\tilde{W}}{^1}\tensor{Y}{^1}\otimes Z^2\tensor{Y}{^2_{\sigma_1\sigma_2}}\epsilon_B(1_{B\tau_1\sigma_2}\tensor{\bar{X}}{^2_{\sigma_1}})\\
			&\otimes1_{B\tau_2}1_{B\tau_3}\tensor{\tilde{X}}{^2}\tensor{\tilde{W}}{^2}\otimes1_{B\tau_4}1_{B\tau_5}1_{B\tau_6}\tensor{X}{^2}\tensor{\bar{W}}{^2}\\
			\stackrel{(B8.1)}{=}\sum&(\tensor{Z}{^1_{\tau_1\tau_2}}\tensor{\bar{X}}{^1_{\tau_3}}\tensor{\tilde{X}}{^1})_{\tau_4}\tensor{X}{^1}\otimes \tensor{\bar{W}}{^1}\tensor{\tilde{W}}{^1}\tensor{Y}{^1}\otimes Z^2\tensor{Y}{^2_{\sigma_1\sigma_2}}\epsilon_B(1_{B\tau_1\sigma_2}\tensor{\bar{X}}{^2_{\sigma_1}})\\
			&\otimes1_{B\tau_2}1_{B\tau_3}\tensor{\tilde{X}}{^2}\tensor{\tilde{W}}{^2}\otimes1_{B\tau_4}\tensor{X}{^2}\tensor{\bar{W}}{^2}\\
			\stackrel{\bm{(**)}}{=}\sum&(\tensor{\tilde{X}}{^1}(\tensor{Z}{^1_{\tau_1}}\tensor{\bar{X}}{^1})_{\sigma_3})_{\tau_4}\tensor{X}{^1}\otimes \tensor{\bar{W}}{^1}\tensor{\tilde{W}}{^1_{\sigma_3}}\tensor{Y}{^1}\otimes Z^2\tensor{Y}{^2_{\sigma_1\sigma_2}}\epsilon_B(1_{B\tau_1\sigma_2}\tensor{\bar{X}}{^2_{\sigma_1}})\\
			&\otimes\tensor{\tilde{X}}{^2}\tensor{\tilde{W}}{^2}\otimes1_{B\tau_4}\tensor{X}{^2}\tensor{\bar{W}}{^2}\\
			\stackrel{C2)}{=}\sum&\tensor{X}{^1}(\tensor{\tilde{X}}{^1}(\tensor{Z}{^1_{\tau_1}}\tensor{\bar{X}}{^1})_{\sigma_3})_{\sigma_4}\epsilon_B(\tensor{W}{^1_{\sigma_4}})\otimes \tensor{\bar{W}}{^1}\tensor{\tilde{W}}{^1_{\sigma_3}}\tensor{Y}{^1}\otimes Z^2\tensor{Y}{^2_{\sigma_1\sigma_2}}\epsilon_B(1_{B\tau_1\sigma_2}\tensor{\bar{X}}{^2_{\sigma_1}})\\
			&\otimes\tensor{\tilde{X}}{^2}\tensor{\tilde{W}}{^2}\otimes\tensor{X}{^2}\tensor{W}{^2}\tensor{\bar{W}}{^2}\\
			\stackrel{C13)}{=}\sum&\tensor{X}{^1}(\tensor{\tilde{X}}{^1}(\tensor{Z}{^1_{\tau_1}}\tensor{\bar{X}}{^1})_{\sigma_3})_{\sigma_4}\epsilon_B(\tensor{\bar{W}}{^1_{(1)\sigma_4}})\otimes \tensor{\bar{W}}{^1_{(2)}}\tensor{\tilde{W}}{^1_{\sigma_3}}\tensor{Y}{^1}\otimes Z^2\tensor{Y}{^2_{\sigma_1\sigma_2}}\\
			&\cdot\epsilon_B(1_{B\tau_1\sigma_2}\tensor{\bar{X}}{^2_{\sigma_1}})\otimes\tensor{\tilde{X}}{^2}\tensor{\tilde{W}}{^2}\otimes\tensor{X}{^2}\tensor{\bar{W}}{^2}\\
			\stackrel{(B7.1)}{=}\sum&\tensor{X}{^1}(\tensor{\tilde{X}}{^1}(\tensor{Z}{^1_{\tau_1}}\tensor{\bar{X}}{^1})_{\sigma_3})_{\sigma_4}\otimes \tensor{\bar{W}}{^1_{\sigma_4}}\tensor{\tilde{W}}{^1_{\sigma_3}}\tensor{Y}{^1}\otimes Z^2\tensor{Y}{^2_{\sigma_1\sigma_2}}\epsilon_B(1_{B\tau_1\sigma_2}\tensor{\bar{X}}{^2_{\sigma_1}})\\
			&\otimes\tensor{\tilde{X}}{^2}\tensor{\tilde{W}}{^2}\otimes\tensor{X}{^2}\tensor{\bar{W}}{^2}\\
			\stackrel{(B7.1)}{=}\sum&\tensor{X}{^1}(\tensor{\tilde{X}}{^1}(\tensor{Z}{^1_{\tau_1}}\tensor{\bar{X}}{^1})_{\sigma_3})_{\sigma_4}\otimes \tensor{\bar{W}}{^1_{\sigma_4}}\epsilon_B(\tensor{\tilde{W}}{^1_{(1)\sigma_3}})\tensor{\tilde{W}}{^1_{(2)}}\tensor{Y}{^1}\otimes Z^2\tensor{Y}{^2_{\sigma_1\sigma_2}}\\
			&\cdot\epsilon_B(1_{B\tau_1\sigma_2}\tensor{\bar{X}}{^2_{\sigma_1}})\otimes\tensor{\tilde{X}}{^2}\tensor{\tilde{W}}{^2}\otimes\tensor{X}{^2}\tensor{\bar{W}}{^2}\\
			\stackrel{C13)}{=}\sum&\tensor{X}{^1}(\tensor{\tilde{X}}{^1}(\tensor{Z}{^1_{\tau_1}}\tensor{\bar{X}}{^1})_{\sigma_3})_{\sigma_4}\otimes \tensor{\bar{W}}{^1_{\sigma_4}}\epsilon_B(\tensor{W}{^1_{\sigma_3}})\tensor{\tilde{W}}{^1}\tensor{Y}{^1}\otimes Z^2\tensor{Y}{^2_{\sigma_1\sigma_2}}\epsilon_B(1_{B\tau_1\sigma_2}\tensor{\bar{X}}{^2_{\sigma_1}})\\
			&\otimes\tensor{\tilde{X}}{^2}\tensor{W}{^2}\tensor{\tilde{W}}{^2}\otimes\tensor{X}{^2}\tensor{\bar{W}}{^2}\\
			\stackrel{C2)}{=}\sum&\tensor{X}{^1}((\tensor{Z}{^1_{\tau_1}}\tensor{\bar{X}}{^1})_{\tau_2}\tensor{\tilde{X}}{^1})_{\sigma_4}\otimes \tensor{\bar{W}}{^1_{\sigma_4}}\tensor{\tilde{W}}{^1}\tensor{Y}{^1}\otimes Z^2\tensor{Y}{^2_{\sigma_1\sigma_2}}\epsilon_B(1_{B\tau_1\sigma_2}\tensor{\bar{X}}{^2_{\sigma_1}})\\
			&\otimes1_{B\tau_2}\tensor{\tilde{X}}{^2}\tensor{\tilde{W}}{^2}\otimes\tensor{X}{^2}\tensor{\bar{W}}{^2}\\
			\stackrel{(B7.1)}{=}\sum&\tensor{X}{^1}(\tensor{Z}{^1_{\tau_1\tau_2}}\tensor{\bar{X}}{^1_{\tau_3}}\tensor{\tilde{X}}{^1})_{\sigma_4}\otimes \tensor{\bar{W}}{^1_{\sigma_4}}\tensor{\tilde{W}}{^1}\tensor{Y}{^1}\otimes Z^2\tensor{Y}{^2_{\sigma_1\sigma_2}}\epsilon_B(1_{B\tau_1\sigma_2}\tensor{\bar{X}}{^2_{\sigma_1}})\\
			&\otimes1_{B\tau_2}1_{B\tau_3}\tensor{\tilde{X}}{^2}\tensor{\tilde{W}}{^2}\otimes\tensor{X}{^2}\tensor{\bar{W}}{^2}\\
			=(id_A&\otimes id_B\otimes id_A\otimes id_B\otimes\epsilon_A\otimes id_B)\Big(\textit{RS of (9)}\Big).
		\end{align*}
		
		Identity corr. to $id_A\otimes id_B\otimes id_A\otimes id_B\otimes id_A\otimes id_B$: 
		\begin{align*}
			\textit{LS of }&\textit{(9)}\\
			=\sum&\tensor{Z}{^1_{\tau_1\tau_2}}\tensor{\bar{X}}{^1_{\tau_3}}X^1\otimes W^1Y^1\otimes (Z^2\tensor{Y}{^2_{\sigma_1\sigma_2}})_{(1)}\epsilon_B(1_{B\tau_1\sigma_2}\tensor{\bar{X}}{^2_{\sigma_1}})\\
			&\otimes(1_{B\tau_2}1_{B\tau_3}X^2W^2)_{(1){\tau_4}}\otimes(Z^2\tensor{Y}{^2_{\sigma_1\sigma_2}})_{(2){\tau_4}}\otimes(1_{B\tau_2}1_{B\tau_3}X^2W^2)_{(2)}\\
			\stackrel{\bm{(***)}}{=}\sum&\tensor{X}{^1}(\tensor{Z}{^1_{\tau_1\tau_2}}\tensor{\bar{X}}{^1_{\tau_3}}\tensor{\tilde{X}}{^1})_{\sigma_6}\otimes \tensor{\bar{W}}{^1_{\sigma_6}}\tensor{\tilde{W}}{^1}\tensor{Y}{^1}\otimes (Z^2\tensor{Y}{^2_{\sigma_1\sigma_2}})_{(1)}\epsilon_B(1_{B\tau_1\sigma_2}\tensor{\bar{X}}{^2_{\sigma_1}})\\
			&\otimes(1_{B\tau_2}1_{B\tau_3}\tensor{\tilde{X}}{^2}\tensor{\tilde{W}}{^2})_{\tau_4}\otimes(Z^2\tensor{Y}{^2_{\sigma_1\sigma_2}})_{(2){\tau_4}}\otimes\tensor{X}{^2}\tensor{\bar{W}}{^2}\\
			\stackrel{(\bigstar)}{=}\sum&\tensor{X}{^1}(\tensor{Z}{^1_{\tau_1}}\tensor{\bar{X}}{^1}(\tensor{\bar{Z}}{^1_{\tau_4\tau_5}}\tensor{\tilde{X}}{^1_{\tau_6}}\tensor{\hat{X}}{^1})_{\sigma_5})_{\sigma_6}\otimes \tensor{\bar{W}}{^1_{\sigma_6}}\tensor{Y}{^1_{\sigma_5}}\tensor{W}{^1}\tensor{\bar{Y}}{^1}\\
			&\otimes \tensor{\bar{Z}}{^2}\tensor{\bar{Y}}{^2_{\sigma_3\sigma_4}}\epsilon_B(1_{B\tau_4\sigma_4}\tensor{\tilde{X}}{^2_{\sigma_3}})\otimes1_{B\tau_5}1_{B\tau_6}\tensor{\hat{X}}{^2}\tensor{W}{^2}\\
			&\otimes\tensor{Z}{^2}\tensor{Y}{^2_{\sigma_1\sigma_2}}\epsilon_B(1_{B\tau_1\sigma_2}\tensor{\bar{X}}{^2_{\sigma_1}})\otimes\tensor{X}{^2}\tensor{\bar{W}}{^2}\\
			\stackrel{(B7.1)}{=}\sum&\tensor{X}{^1}(\tensor{Z}{^1_{\tau_1}}\tensor{\bar{X}}{^1}(\tensor{\bar{Z}}{^1_{\tau_4\tau_5}}\tensor{\tilde{X}}{^1_{\tau_6}}\tensor{\hat{X}}{^1})_{\sigma_5})_{\sigma_6}\otimes \epsilon_B(\tensor{\bar{W}}{^1_{(1)\sigma_6}})\tensor{\bar{W}}{^1_{(2)}}\tensor{Y}{^1_{\sigma_5}}\tensor{W}{^1}\tensor{\bar{Y}}{^1}\\
			&\otimes \tensor{\bar{Z}}{^2}\tensor{\bar{Y}}{^2_{\sigma_3\sigma_4}}\epsilon_B(1_{B\tau_4\sigma_4}\tensor{\tilde{X}}{^2_{\sigma_3}})\otimes1_{B\tau_5}1_{B\tau_6}\tensor{\hat{X}}{^2}\tensor{W}{^2}\\
			&\otimes\tensor{Z}{^2}\tensor{Y}{^2_{\sigma_1\sigma_2}}\epsilon_B(1_{B\tau_1\sigma_2}\tensor{\bar{X}}{^2_{\sigma_1}})\otimes\tensor{X}{^2}\tensor{\bar{W}}{^2}\\
			\stackrel{C13)}{=}\sum&\tensor{X}{^1}(\tensor{Z}{^1_{\tau_1}}\tensor{\bar{X}}{^1}(\tensor{\bar{Z}}{^1_{\tau_4\tau_5}}\tensor{\tilde{X}}{^1_{\tau_6}}\tensor{\hat{X}}{^1})_{\sigma_5})_{\sigma_6}\otimes \epsilon_B(\tensor{\bar{\bar{W}}}{^1_{\sigma_6}})\tensor{\bar{W}}{^1}\tensor{Y}{^1_{\sigma_5}}\tensor{W}{^1}\tensor{\bar{Y}}{^1}\\
			&\otimes \tensor{\bar{Z}}{^2}\tensor{\bar{Y}}{^2_{\sigma_3\sigma_4}}\epsilon_B(1_{B\tau_4\sigma_4}\tensor{\tilde{X}}{^2_{\sigma_3}})\otimes1_{B\tau_5}1_{B\tau_6}\tensor{\hat{X}}{^2}\tensor{W}{^2}\\
			&\otimes\tensor{Z}{^2}\tensor{Y}{^2_{\sigma_1\sigma_2}}\epsilon_B(1_{B\tau_1\sigma_2}\tensor{\bar{X}}{^2_{\sigma_1}})\otimes\tensor{X}{^2}\tensor{\bar{\bar{W}}}{^2}\tensor{\bar{W}}{^2}\\
			\stackrel{C2)}{=}\sum&(\tensor{Z}{^1_{\tau_1}}\tensor{\bar{X}}{^1}(\tensor{\bar{Z}}{^1_{\tau_4\tau_5}}\tensor{\tilde{X}}{^1_{\tau_6}}\tensor{\hat{X}}{^1})_{\sigma_5})_{\tau_7}\tensor{X}{^1}\otimes \tensor{\bar{W}}{^1}\tensor{Y}{^1_{\sigma_5}}\tensor{W}{^1}\tensor{\bar{Y}}{^1}\\
			&\otimes \tensor{\bar{Z}}{^2}\tensor{\bar{Y}}{^2_{\sigma_3\sigma_4}}\epsilon_B(1_{B\tau_4\sigma_4}\tensor{\tilde{X}}{^2_{\sigma_3}})\otimes1_{B\tau_5}1_{B\tau_6}\tensor{\hat{X}}{^2}\tensor{W}{^2}\\
			&\otimes\tensor{Z}{^2}\tensor{Y}{^2_{\sigma_1\sigma_2}}\epsilon_B(1_{B\tau_1\sigma_2}\tensor{\bar{X}}{^2_{\sigma_1}})\otimes1_{B\tau_7}\tensor{X}{^2}\tensor{\bar{W}}{^2}\\
			\stackrel{(B8.1)}{=}\sum&(\tensor{Z}{^1_{\tau_1}}\tensor{\bar{X}}{^1})_{\tau_7}(\tensor{\bar{Z}}{^1_{\tau_4\tau_5}}\tensor{\tilde{X}}{^1_{\tau_6}}\tensor{\hat{X}}{^1})_{\sigma_5\tau_8}\tensor{X}{^1}\otimes \tensor{\bar{W}}{^1}\tensor{Y}{^1_{\sigma_5}}\tensor{W}{^1}\tensor{\bar{Y}}{^1}\\
			&\otimes \tensor{\bar{Z}}{^2}\tensor{\bar{Y}}{^2_{\sigma_3\sigma_4}}\epsilon_B(1_{B\tau_4\sigma_4}\tensor{\tilde{X}}{^2_{\sigma_3}})\otimes1_{B\tau_5}1_{B\tau_6}\tensor{\hat{X}}{^2}\tensor{W}{^2}\\
			&\otimes\tensor{Z}{^2}\tensor{Y}{^2_{\sigma_1\sigma_2}}\epsilon_B(1_{B\tau_1\sigma_2}\tensor{\bar{X}}{^2_{\sigma_1}})\otimes1_{B\tau_7}1_{B\tau_8}\tensor{X}{^2}\tensor{\bar{W}}{^2}\\
			\stackrel{C2)}{=}\sum&(\tensor{Z}{^1_{\tau_1}}\tensor{\bar{X}}{^1})_{\tau_7}\tensor{X}{^1}(\tensor{\bar{Z}}{^1_{\tau_4\tau_5}}\tensor{\tilde{X}}{^1_{\tau_6}}\tensor{\hat{X}}{^1})_{\sigma_5\sigma_6}\epsilon_B(\tensor{\bar{\bar{W}}}{^1_{\sigma_6}})\otimes \tensor{\bar{W}}{^1}\tensor{Y}{^1_{\sigma_5}}\tensor{W}{^1}\tensor{\bar{Y}}{^1}\\
			&\otimes \tensor{\bar{Z}}{^2}\tensor{\bar{Y}}{^2_{\sigma_3\sigma_4}}\epsilon_B(1_{B\tau_4\sigma_4}\tensor{\tilde{X}}{^2_{\sigma_3}})\otimes1_{B\tau_5}1_{B\tau_6}\tensor{\hat{X}}{^2}\tensor{W}{^2}\\
			&\otimes\tensor{Z}{^2}\tensor{Y}{^2_{\sigma_1\sigma_2}}\epsilon_B(1_{B\tau_1\sigma_2}\tensor{\bar{X}}{^2_{\sigma_1}})\otimes1_{B\tau_7}\tensor{X}{^2}\tensor{\bar{\bar{W}}}{^2}\tensor{\bar{W}}{^2}\\
			\stackrel{C13)}{=}\sum&(\tensor{Z}{^1_{\tau_1}}\tensor{\bar{X}}{^1})_{\tau_7}\tensor{X}{^1}(\tensor{\bar{Z}}{^1_{\tau_4\tau_5}}\tensor{\tilde{X}}{^1_{\tau_6}}\tensor{\hat{X}}{^1})_{\sigma_5\sigma_6}\epsilon_B(\tensor{\bar{W}}{^1_{(1)\sigma_6}})\otimes \tensor{\bar{W}}{^1_{(2)}}\tensor{Y}{^1_{\sigma_5}}\tensor{W}{^1}\tensor{\bar{Y}}{^1}\\
			&\otimes \tensor{\bar{Z}}{^2}\tensor{\bar{Y}}{^2_{\sigma_3\sigma_4}}\epsilon_B(1_{B\tau_4\sigma_4}\tensor{\tilde{X}}{^2_{\sigma_3}})\otimes1_{B\tau_5}1_{B\tau_6}\tensor{\hat{X}}{^2}\tensor{W}{^2}\\
			&\otimes\tensor{Z}{^2}\tensor{Y}{^2_{\sigma_1\sigma_2}}\epsilon_B(1_{B\tau_1\sigma_2}\tensor{\bar{X}}{^2_{\sigma_1}})\otimes1_{B\tau_7}\tensor{X}{^2}\tensor{\bar{W}}{^2}\\
			\stackrel{(B7.1)}{=}\sum&(\tensor{Z}{^1_{\tau_1}}\tensor{\bar{X}}{^1})_{\tau_7}\tensor{X}{^1}(\tensor{\bar{Z}}{^1_{\tau_4\tau_5}}\tensor{\tilde{X}}{^1_{\tau_6}}\tensor{\hat{X}}{^1})_{\sigma_5\sigma_6}\otimes \tensor{\bar{W}}{^1_{\sigma_6}}\tensor{Y}{^1_{\sigma_5}}\tensor{W}{^1}\tensor{\bar{Y}}{^1}\\
			&\otimes \tensor{\bar{Z}}{^2}\tensor{\bar{Y}}{^2_{\sigma_3\sigma_4}}\epsilon_B(1_{B\tau_4\sigma_4}\tensor{\tilde{X}}{^2_{\sigma_3}})\otimes1_{B\tau_5}1_{B\tau_6}\tensor{\hat{X}}{^2}\tensor{W}{^2}\\
			&\otimes\tensor{Z}{^2}\tensor{Y}{^2_{\sigma_1\sigma_2}}\epsilon_B(1_{B\tau_1\sigma_2}\tensor{\bar{X}}{^2_{\sigma_1}})\otimes1_{B\tau_7}\tensor{X}{^2}\tensor{\bar{W}}{^2}\\
			\stackrel{(1.2)}{=}\sum&(\tensor{Z}{^1_{\tau_1}}\tensor{\bar{X}}{^1})_{\tau_7}\tensor{X}{^1}(\tensor{\bar{Z}}{^1_{\tau_4\tau_5}}\tensor{\tilde{X}}{^1_{\tau_6}}\tensor{\hat{X}}{^1})_{\sigma_5}\otimes (\tensor{\bar{W}}{^1}\tensor{Y}{^1})_{\sigma_5}\tensor{W}{^1}\tensor{\bar{Y}}{^1}\\
			&\otimes \tensor{\bar{Z}}{^2}\tensor{\bar{Y}}{^2_{\sigma_3\sigma_4}}\epsilon_B(1_{B\tau_4\sigma_4}\tensor{\tilde{X}}{^2_{\sigma_3}})\otimes1_{B\tau_5}1_{B\tau_6}\tensor{\hat{X}}{^2}\tensor{W}{^2}\\
			&\otimes\tensor{Z}{^2}\tensor{Y}{^2_{\sigma_1\sigma_2}}\epsilon_B(1_{B\tau_1\sigma_2}\tensor{\bar{X}}{^2_{\sigma_1}})\otimes1_{B\tau_7}\tensor{X}{^2}\tensor{\bar{W}}{^2}\\
			\stackrel{(B8.1)}{=}\sum&\tensor{Z}{^1_{\tau_1\tau_2}}\tensor{\bar{X}}{^1_{\tau_3}}\tensor{X}{^1}(\tensor{\bar{Z}}{^1_{\tau_4\tau_5}}\tensor{\tilde{X}}{^1_{\tau_6}}\tensor{\hat{X}}{^1})_{\sigma_5}\otimes (\tensor{\bar{W}}{^1}\tensor{Y}{^1})_{\sigma_5}\tensor{W}{^1}\tensor{\bar{Y}}{^1}\\
			&\otimes \tensor{\bar{Z}}{^2}\tensor{\bar{Y}}{^2_{\sigma_3\sigma_4}}\epsilon_B(1_{B\tau_4\sigma_4}\tensor{\tilde{X}}{^2_{\sigma_3}})\otimes1_{B\tau_5}1_{B\tau_6}\tensor{\hat{X}}{^2}\tensor{W}{^2}\\
			&\otimes\tensor{Z}{^2}\tensor{Y}{^2_{\sigma_1\sigma_2}}\epsilon_B(1_{B\tau_1\sigma_2}\tensor{\bar{X}}{^2_{\sigma_1}})\otimes1_{B\tau_2}1_{B\tau_3}\tensor{X}{^2}\tensor{\bar{W}}{^2}\\
			=\textit{RS of}&\textit{ (9)}.
		\end{align*}
		In above $(\bigstar)$ represents the identity corresponding to $id_A\otimes id_B\otimes id_A\otimes id_B\otimes id_A\otimes\epsilon_B$ which is derived in another line. That is
		\begin{align*}
			\sum&\tensor{Z}{^1_{\tau_1\tau_2}}\tensor{\bar{X}}{^1_{\tau_3}}\tensor{X}{^1}\otimes\tensor{W}{^1}\tensor{Y}{^1}\otimes(\tensor{Z}{^2}\tensor{Y}{^2_{\sigma}})_{(1)}\epsilon_B((1_{B\tau_1}\tensor{\bar{X}}{^2})_{\sigma})\otimes(1_{B\tau_2}1_{B\tau_3}\tensor{X}{^2}\tensor{W}{^2})_{\tau_4}\\
			&\otimes(\tensor{Z}{^2}\tensor{Y}{^2_{\sigma}})_{(2)\tau_4}\\
			=\sum&\tensor{Z}{^1_{\tau_1}}\tensor{X}{^1}(\tensor{\bar{Z}}{^1_{\tau_3\tau_4}}\tensor{\bar{X}}{^1_{\tau_2}}\tensor{\tilde{\tilde{X}}}{^1})_{\sigma_4}\otimes\tensor{Y}{^1_{\sigma_4}}\tensor{W}{^1}\tensor{\bar{Y}}{^1}\otimes\tensor{\bar{Z}}{^2}\tensor{\bar{Y}}{^2_{\sigma_1}}\epsilon_B((1_{B\tau_3}\tensor{\bar{X}}{^2})_{\sigma_1})\\
			&\otimes1_{B\tau_4}1_{B\tau_2}\tensor{\tilde{\tilde{X}}}{^2}\tensor{W}{^2}\otimes\tensor{Z}{^2}\tensor{Y}{^2_{\sigma_6}}\epsilon_B((1_{B\tau_1}\tensor{X}{^2})_{\sigma_6}).
		\end{align*}
	\end{proof}
	
	Combining Proposition \ref{thm:3.4}, Proposition \ref{thm:3.5} and Theorem \ref{thm:3.6}, we conclude that:
	
	\begin{thm}\label{thm:3.7}
	Let $A{_\tau\times_\sigma}B$ be a smash biproduct bialgebra. Under condition that $\sigma$ is right conormal, $A{_\tau\times_\sigma}B$ admits a quasitriangular structure if and only if there exist normalized elements $W=\sum W^1\otimes W^2\in B\otimes B$, $X=\sum X^1\otimes X^2\in A\otimes B$, $Y=\sum Y^1\otimes Y^2\in B\otimes A$ and $Z=\sum Z^1\otimes Z^2\in A\otimes A$ satisfying $C1)-C19)$ above. In this case, the quasitriangular structure is given as $\alpha=\sum \tensor{Z}{^1_{\tau_1\tau_2}}\tensor{\bar{X}}{^1_{\tau_3}}X^1\otimes W^1Y^1\otimes Z^2\tensor{Y}{^2_{\sigma_1\sigma_2}}\epsilon_B(1_{B\tau_1\sigma_2}\tensor{\bar{X}}{^2_{\sigma_1}})\otimes1_{B\tau_2}1_{B\tau_3}X^2W^2$.
	\end{thm}

	\section{Applications}\label{Sec:4}
	In this section, we show that by adding extra normal or conormal conditions, we can get various simplified factorization theories. These results fit in with different constructions of Hopf algebras.

	\subsection{When $\tau$ is right normal}\label{Sec:4.1}
	Let $A{_\tau\times_\sigma}B$ be a smash biproduct bialgebra for which $\sigma$ is right conormal and additionally $\tau$ is right normal. Then $i_B:B\rightarrow A{_\tau\times_\sigma}B$ $(b\mapsto1_A\otimes b)$ and $\pi_B:A{_\tau\times_\sigma}B\rightarrow B\quad(\sum_i a_i\otimes b_i\mapsto\sum_i\epsilon_A(a_i)b_i)$ are bialgebra maps. 
	
	Using the extra condition that $\tau$ is right normal: $\sum b_\tau\otimes 1_{A\tau}=b\otimes 1_A$, we have that $\sum b_{\tau}\otimes a_{\tau}=1_{B\tau}b\otimes a_{\tau}$ for $(B5)$. Now the formula in Proposition \ref{thm:3.4} becomes:
	\begin{align*}
		&\sum\tensor{Z}{^1_{\tau_1\tau_2}}\tensor{\bar{X}}{^1_{\tau_3}}X^1\otimes W^1Y^1\otimes Z^2\epsilon_B(1_{B\tau_1\sigma_2}\tensor{\bar{X}}{^2_{\sigma_1}})\tensor{Y}{^2_{\sigma_1\sigma_2}}\otimes1_{B\tau_2}1_{B\tau_3}X^2W^2\hspace{10em}\\
		=&\sum\tensor{Z}{^1_{\tau_1\tau_2}}\tensor{\bar{X}}{^1_{\tau_3}}X^1\otimes W^1Y^1\otimes Z^2\epsilon_B((1_{B\tau_1}\tensor{\bar{X}}{^2})_{\sigma})\tensor{Y}{^2_{\sigma}}\otimes1_{B\tau_2}1_{B\tau_3}X^2W^2\\
		=&\sum\tensor{Z}{^1_{\tau_1\tau_2}}\tensor{\bar{X}}{^1_{\tau_3}}X^1\otimes W^1Y^1\otimes Z^2\epsilon_B(\tensor{\bar{X}}{^2_{\tau_1\sigma}})\tensor{Y}{^2_{\sigma}}\otimes1_{B\tau_2}1_{B\tau_3}X^2W^2\\
		=&\sum\tensor{Z}{^1_{\tau_1\tau_2}}\tensor{\bar{X}}{^1_{\tau_3}}X^1\otimes W^1Y^1\otimes Z^2\epsilon_B(\tensor{\bar{X}}{^2_{\tau_1\sigma}})\tensor{Y}{^2_{\sigma}}\otimes\tensor{X}{^2_{\tau_3\tau_2}}W^2.
	\end{align*}
	
	Also the identities in Proposition \ref{thm:3.5} becomes:
	\begin{align*}
		C1')\quad\sum&{\tensor{Z}{^1_{\tau_1\tau_2}}}{\tensor{\bar{X}}{^1_{\tau_3}}} X^1a_{(1)\sigma_2}\epsilon_B( (W^1Y^1)_{\sigma_2})\otimes Z^2{\tensor{Y}{^2_{\sigma_1}}}\epsilon_B(\tensor{\bar{X}}{^2_{\tau_1\sigma_1}})a_{(2){\sigma_3}}\\
		&\cdot\epsilon_B(({\tensor{X}{^2_{\tau_3\tau_2}}}W^2)_{\sigma_3})=\sum a_{(2)\tau}Z^1\otimes a_{(1)}\tensor{Z}{^2_\sigma}\epsilon_B( 1_{B\tau\sigma});\\
		C2')\quad\sum&X^1a_\sigma \epsilon_B(\tensor{W}{^1_\sigma})\otimes X^2W^2=\sum a_\tau X^1\otimes\tensor{X}{^2_\tau};\\
		C3')\quad\sum&W^1Y^11_{B\tau}\otimes Y^2a_{\tau\sigma}\epsilon_B(\tensor{W}{^2_\sigma})=\sum Y^1\otimes aY^2;\\
		C4')\quad\sum&Z^1\otimes Z^2\epsilon_B(b)=\sum\tensor{Z}{^1_{\sigma_1}}\otimes\tensor{Z}{^2_{\sigma_2}}\epsilon_B(b_{\sigma_2\sigma_1});\\
		C5')\quad\sum&X^1\otimes X^2b=\sum\tensor{X}{^1_\sigma}\epsilon_B(b_{(2)\sigma})\otimes b_{(1)}X^2;\\
		C6')\quad\sum&Y^1b\otimes Y^2=\sum b_\sigma Y^1\otimes\tensor{Y}{^2_\sigma};\\
		C7')\quad\sum&W^1b_{(1)}\otimes W^2b_{(2)}=\sum b_{(2)}W^1\otimes b_{(1)}W^2;\\
		C8')\quad\sum&\tensor{Z}{^1_{(1)}}\otimes \tensor{Z}{^1_{(2)}}\otimes Z^2=\sum \tensor{Z}{^1_\tau}X^1\otimes\bar{Z}^1\otimes Z^2\tensor{\bar{Z}}{^2_\sigma}\epsilon_B(\tensor{X}{^2_{\tau\sigma}});\\
		C9')\quad\sum&\tensor{X}{^1_{(1)}}\otimes \tensor{X}{^1_{(2)}}\otimes X^2=\sum X^1\otimes{\bar{X}}^1\otimes X^2{\bar{X}}^2;\\
		C10')\quad\sum&\tensor{Y}{^1_{\tau_1}}\otimes (\tensor{Z}{^1_{\tau_2}}X^1)_{\tau_1}\otimes Z^2\tensor{Y}{^2_\sigma}\epsilon_B(\tensor{X}{^2_{\tau_2\sigma}})=\sum W^1Y^1\otimes Z^1\\
		&\otimes Y^2\tensor{Z}{^2_\sigma}\epsilon_B(\tensor{W}{^2_\sigma});\\
		C11')\quad\sum&\tensor{W}{^1_\tau}\otimes \tensor{X}{^1_\tau}\otimes X^2W^2=\sum W^1\otimes X^1\otimes W^2X^2;\\
		C12')\quad\sum&\tensor{Y}{^1_{(1)}}\otimes \tensor{Y}{^1_{(2)}}\otimes Y^2=\sum W^1Y^1\otimes {\bar{Y}}^1\otimes Y^2\tensor{\bar{Y}}{^2_\sigma}\epsilon_B(\tensor{W}{^2_\sigma});\\
		C13')\quad\sum&\tensor{W}{^1_{(1)}}\otimes \tensor{W}{^1_{(2)}}\otimes W^2=\sum W^1\otimes{\bar{W}}^1\otimes W^2{\bar{W}}^2;\\
		C14')\quad\sum&Z^1\otimes\tensor{Z}{^2_{(1)}}\otimes \tensor{Z}{^2_{(2)}}=\sum \tensor{Z}{^1_\tau}X^1\tensor{\bar{Z}}{^1_{\sigma_1}}\epsilon_B(\tensor{Y}{^1_{\sigma_1}})\otimes{\bar{Z}}^2\\
		&\otimes Z^2\tensor{Y}{^2_{\sigma_2}}\epsilon_B(\tensor{X}{^2_{\tau\sigma_2}});\\
		C15')\quad\sum&\tensor{Z}{^1_{\tau_1}} X^1\otimes \tensor{X}{^2_{\tau_1\tau_2}}\otimes\tensor{Z}{^2_{\tau_2}}=\sum \tensor{Z}{^1_\tau}X^1\tensor{\bar{X}}{^1_{\sigma_1}}\epsilon_B(\tensor{Y}{^1_{\sigma_1}})\otimes {\bar{X}}^2\\
		&\otimes Z^2\tensor{Y}{^2_{\sigma_2}}\epsilon_B(\tensor{X}{^2_{\tau\sigma_2}});\\
		C16')\quad\sum&X^1\otimes\tensor{X}{^2_{(1)}}\otimes \tensor{X}{^2_{(2)}}=\sum X^1 \tensor{\bar{X}}{^1_\sigma}\epsilon_B(\tensor{W}{^1_\sigma})\otimes{\bar{X}}^2\otimes X^2W^2;\\
		C17')\quad\sum&Y^1\otimes\tensor{Y}{^2_{(1)}}\otimes \tensor{Y}{^2_{(2)}}=\sum Y^1{\bar{Y}}^1\otimes {\bar{Y}}^2\otimes Y^2;\\
		C18')\quad\sum&W^1Y^1\otimes\tensor{W}{^2_\tau}\otimes \tensor{Y}{^2_\tau}=\sum Y^1W^1 \otimes W^2\otimes Y^2;\\
		C19')\quad\sum&W^1\otimes\tensor{W}{^2_{(1)}}\otimes \tensor{W}{^2_{(2)}}=\sum W^1{\bar{W}}^1 \otimes {\bar{W}}^2\otimes W^2.
	\end{align*}
	
	Thus from Theorem \ref{thm:3.7}, we have
	\begin{coro}\label{thm:4.1}
		Let $A{_\tau\times_\sigma}B$ be a smash biproduct bialgebra. Under conditions that $\sigma$ is right conormal and $\tau$ is right normal, $A{_\tau\times_\sigma}B$ admits a quasitriangular structure if and only if there exist normalized elements $W=\sum W^1\otimes W^2\in B\otimes B$, $X=\sum X^1\otimes X^2\in A\otimes B$, $Y=\sum Y^1\otimes Y^2\in B\otimes A$ and $Z=\sum Z^1\otimes Z^2\in A\otimes A$ satisfying  $C1')-C19')$ above. In this case, the quasitriangular structure is given as $\alpha=\sum{\tensor{Z}{^1_{\tau_1\tau_2}}}{\tensor{\bar{X}}{^1_{\tau_3}}} X^1\otimes W^1Y^1\otimes Z^2{\tensor{Y}{^2_\sigma}}\epsilon_B(\tensor{\bar{X}}{^2_{\tau_1\sigma}})\otimes{\tensor{X}{^2_{\tau_3\tau_2}}}W^2$.
	\end{coro}
	\begin{proof}
		Using $\sum b_\tau\otimes 1_{A\tau}=b\otimes 1_A$ in Theorem \ref{thm:3.7}.
	\end{proof}
	
	When the smash biproduct bialgebra $A{_\tau\times_\sigma}B$ is moreover a Hopf algebra, this corollary is still true and it's the accurate form of \cite[Theorem 3.13]{M&W} because the latter omitted the right normal and right conormal conditions in the expression of the theorem. Actually if we remove the right normal and right conormal conditions, the existence of such $W\in B\otimes B$, $X\in A\otimes B$, $Y\in B\otimes A$ and $Z\in A\otimes A$ is not necessary to make $A{_\tau\times_\sigma}B$ quasitriangular. Thus Corollary \ref{thm:4.1} above corrects \cite[Theorem 3.13]{M&W} and generalizes it to smash biproduct bialgebras.
	
	Moreover we could show that the theories in \cite{Z&Z} and \cite{M&W} are equivalent. Suppose that $A{_\tau\times_\sigma}B$ is a smash biproduct Hopf algebra for which $\sigma$ is right conormal and $\tau$ is right normal. Now define
	\begin{align*}
		\triangleright:B\otimes A\rightarrow A\quad&(b\otimes a\mapsto b\triangleright a=\sum a_{\sigma}\epsilon_B(b_{\sigma})),\\
		\rho:A\rightarrow B\otimes A\quad&(a\mapsto\sum a_{[-1]}\otimes a_{[0]}=\sum1_{B\tau}\otimes a_{\tau}).
	\end{align*}
	
	Then $A$ is a left $B$-module and left $B$-comodule. Moreover we have that $A$ is a bialgebra in the category of Yetter-Drinfeld modules $_B^B{\mathcal{YD}}$ and the Radford's biproduct $A\times B$ coincides with $A{_\tau\times_\sigma}B$.
	
	On the other hand, suppose that $B$ is a Hopf algebra and $A\in{_B^B{\mathcal{YD}}}$ a braided Hopf algebra. By setting
	\begin{align*}
		&\sigma:B\otimes A\rightarrow A\otimes B\quad (b\otimes a\mapsto\sum b_{(1)}\triangleright a\otimes b_{(2)}),\\
		&\tau:A\otimes B\rightarrow B\otimes A\quad(a\otimes b\mapsto\sum a_{[-1]}b\otimes a_{[0]}),
	\end{align*} 
	the Radford's biproduct Hopf algebra $A\times B$ is a smash biproduct Hopf algebra $A{_\tau\times_\sigma}B$ for which $\sigma$ is right conormal and $\tau$ is right normal. Hence these smash biproduct Hopf algebras satisfying both the right normal and right conormal conditions are exactly Radford's biproduct Hopf algebras. In particular, we have
	\begin{align*}
		&\sum{\tensor{Z}{^1_{\tau_1\tau_2}}}{\tensor{\bar{X}}{^1_{\tau_3}}} X^1\otimes W^1Y^1\otimes Z^2{\tensor{Y}{^2_\sigma}}\epsilon_B(\tensor{\bar{X}}{^2_{\tau_1\sigma}})\otimes{\tensor{X}{^2_{\tau_3\tau_2}}}W^2\\
		=&\sum{\tensor{Z}{^1_{\tau_1\tau_2}}}{\tensor{\bar{X}}{^1_{\tau_3}}} X^1\otimes W^1Y^1\otimes Z^2{\tensor{Y}{^2_\sigma}}\epsilon_B((1_{B\tau_1}\tensor{\bar{X}}{^2})_{\sigma})\otimes{1_{B\tau_2}1_{B\tau_3}\tensor{X}{^2}}W^2\\
		=&\sum(\tensor{Z}{^1_{[0]}})_{[0]}\tensor{\bar{X}}{^1_{[0]}}X^1\otimes W^1Y^1\otimes Z^2((\tensor{Z}{^1_{[-1]}}\tensor{\bar{X}}{^2})\triangleright\tensor{Y}{^2})\otimes(\tensor{Z}{^1_{[0]}})_{[-1]}\tensor{\bar{X}}{^1_{[-1]}}X^2W^2\\
		=&\sum\tensor{Z}{^1_{[0]}}\tensor{\bar{X}}{^1_{[0]}} X^1\otimes W^1Y^1\otimes Z^2((\tensor{Z}{^1_{[-2]}}\tensor{\bar{X}}{^2})\triangleright\tensor{Y}{^2})\otimes\tensor{Z}{^1_{[-1]}}\tensor{\bar{X}}{^1_{[-1]}}X^2W^2
	\end{align*}
	which is the form of the reconstruction formula in \cite{Z&Z}.
	
	We specially note that the identities above are classified into several groups in \cite{M&W}. $C2')$, $C3')$, $C5')$, $C6')$, $C4')$, $C11')$, $C18')$, $C15')$ and $C10')$ here are respectively the conditions $C1)-C9)$ of \cite[Proposition 3.4]{M&W}, and $C9')$, $C16')$ the compatible triple conditions for $X$ (cf. \cite[Definition 3.5]{M&W}), $C12')$, $C17')$ the compatible triple conditions for $Y$ (cf. \cite[Definition 3.7]{M&W}), $C8')$, $C14')$, $C1')$ the weak quasitriangular structure conditions for $Z$ (cf. \cite[Definition 3.8]{M&W}) and $C13')$, $C19')$, $C7')$ the quasitriangular structure conditions for $W$.

	\subsection{When $\tau$ is left normal}\label{Sec:4.2}
	Let $A{_\tau\times_\sigma}B$ be a smash biproduct bialgebra for which $\sigma$ is right conormal and additionally $\tau$ is left normal. 
	
	Since $\tau$ is left normal: $\sum1_{B\tau}\otimes a_\tau=1_B\otimes a$, the formula in Proposition \ref{thm:3.4} now becomes:
	\begin{align*}
		&\sum \tensor{Z}{^1_{\tau_1\tau_2}}\tensor{\bar{X}}{^1_{\tau_3}}X^1\otimes W^1Y^1\otimes Z^2\epsilon_B(1_{B\tau_1\sigma_2}\tensor{\bar{X}}{^2_{\sigma_1}})\tensor{Y}{^2_{\sigma_1\sigma_2}}\otimes1_{B\tau_2}1_{B\tau_3}X^2W^2\\
		=&\sum \tensor{Z}{^1}\tensor{\bar{X}}{^1}X^1\otimes W^1Y^1\otimes Z^2\epsilon_B(\tensor{\bar{X}}{^2_\sigma})\tensor{Y}{^2_\sigma}\otimes X^2W^2.
	\end{align*}
	
	And the identities in Proposition \ref{thm:3.5} become:
	\begin{align*}
		C1'')\quad\sum& Z^1{\bar{X}}^1X^1a_{(1)\sigma_2}\epsilon_B((W^1Y^1)_{\sigma_2})\otimes Z^2\tensor{Y}{^2_{\sigma_1}}\epsilon_B(\tensor{\bar{X}}{^2_{\sigma_1}})a_{(2)\sigma_3}\epsilon_B((X^2W^2)_{\sigma_3})\\
		=\sum&a_{(2)}Z^1\otimes a_{(1)}Z^2;\\
		C2'')\quad\sum&X^1a_{\sigma}\epsilon_B(\tensor{W}{^1_{\sigma}})\otimes X^2W^2=\sum aX^1\otimes X^2;\\
		C3'')\quad\sum&W^1Y^1\otimes Y^2a_\sigma\epsilon_B(\tensor{W}{^2_\sigma})=\sum Y^1\otimes aY^2;\\
		C4'')\quad\sum&Z^1\otimes Z^2\epsilon_B(b)=\sum1_{A\tau}\tensor{Z}{^1_{\sigma_1}}\epsilon_B( b_{(2)\sigma_1})\otimes \tensor{Z}{^2_{\sigma_2}}\epsilon_B(b_{(1)\tau\sigma_2});\\
		C5'')\quad\sum&X^1\otimes X^2b=\sum1_{A\tau}\tensor{X}{^1_\sigma}\epsilon_B( b_{(2)\sigma})\otimes b_{(1)\tau}X^2;\\
		C6'')\quad\sum&W^1Y^1b_\tau\otimes Y^21_{A\tau\sigma}\epsilon_B(\tensor{W}{^2_\sigma})=\sum b_{\sigma}Y^1\otimes\tensor{Y}{^2_\sigma};\\
		C7'')\quad\sum&W^1b_{(1)}\otimes W^2b_{(2)}=\sum b_{(2)}W^1\otimes b_{(1)}W^2;\\
		C8'')\quad\sum&\tensor{Z}{^1_{(1)}}\otimes\tensor{Z}{^1_{(2)}}\otimes Z^2=\sum Z^1X^1\otimes {\bar{Z}}^1\otimes Z^2\tensor{\bar{Z}}{^2_\sigma}\epsilon_B(\tensor{X}{^2_\sigma});\\
		C9'')\quad\sum&\tensor{X}{^1_{(1)}}\otimes\tensor{X}{^1_{(2)}}\otimes X^2=\sum X^1\otimes {\bar{X}}^1\otimes X^2{\bar{X}}^2;\\
		C10'')\quad\sum&\tensor{Y}{^1_\tau}\otimes(Z^1X^1)_\tau\otimes Z^2\tensor{Y}{^2_\sigma}\epsilon_B(\tensor{X}{^2_\sigma})=\sum W^1Y^1\otimes Z^1\otimes Y^2\tensor{Z}{^2_\sigma}\epsilon_B(\tensor{W}{^2_\sigma})\\
		C11'')\quad\sum&\tensor{W}{^1_\tau}\otimes\tensor{X}{^1_\tau}\otimes X^2W^2=\sum W^1\otimes X^1\otimes W^2X^2;\\
		C12'')\quad\sum&\tensor{Y}{^1_{(1)}}\otimes\tensor{Y}{^1_{(2)}}\otimes Y^2=\sum W^1Y^1\otimes {\bar{Y}}^1\otimes Y^2\tensor{\bar{Y}}{^2_\sigma}\epsilon_B(\tensor{W}{^2_\sigma});\\
		C13'')\quad\sum&\tensor{W}{^1_{(1)}}\otimes \tensor{W}{^1_{(2)}}\otimes W^2=\sum W^1\otimes{\bar{W}}^1\otimes W^2{\bar{W}}^2;\\
		C14'')\quad\sum&Z^1\otimes\tensor{Z}{^2_{(1)}}\otimes \tensor{Z}{^2_{(2)}}=\sum \tensor{Z}{^1}X^1\tensor{\bar{Z}}{^1_{\sigma_1}}\epsilon_B(\tensor{Y}{^1_{\sigma_1}})\otimes{\bar{Z}}^2\otimes Z^2\tensor{Y}{^2_{\sigma_2}}\epsilon_B(\tensor{X}{^2_{\sigma_2}});\\
		C15'')\quad\sum&\tensor{Z}{^1}X^1\otimes \tensor{X}{^2_\tau}\otimes\tensor{Z}{^2_\tau}=\sum \tensor{Z}{^1}X^1\tensor{\bar{X}}{^1_{\sigma_1}}\epsilon_B(\tensor{Y}{^1_{\sigma_1}})\otimes {\bar{X}}^2\otimes Z^2\tensor{Y}{^2_{\sigma_2}}\epsilon_B(\tensor{X}{^2_{\sigma_2}});\\
		C16'')\quad\sum&X^1\otimes\tensor{X}{^2_{(1)}}\otimes \tensor{X}{^2_{(2)}}=\sum X^1 \tensor{\bar{X}}{^1}\otimes X^2\otimes{\bar{X}}^2;\\
		C17'')\quad\sum&Y^1\otimes\tensor{Y}{^2_{(1)}}\otimes \tensor{Y}{^2_{(2)}}=\sum Y^1{\bar{Y}}^1\otimes {\bar{Y}}^2\otimes Y^2;\\
		C18'')\quad\sum&W^1Y^1\otimes\tensor{W}{^2_\tau}\otimes \tensor{Y}{^2_\tau}=\sum Y^1W^1 \otimes W^2\otimes Y^2;\\
		C19'')\quad\sum&W^1\otimes\tensor{W}{^2_{(1)}}\otimes \tensor{W}{^2_{(2)}}=\sum W^1{\bar{W}}^1 \otimes {\bar{W}}^2\otimes W^2.
	\end{align*}
	
	Thus from Theorem \ref{thm:3.7} we have:
	\begin{coro}\label{thm:4.2}
		Let $A{_\tau\times_\sigma}B$ be a smash biproduct bialgebra. Under conditions that $\sigma$ is right conormal and $\tau$ is left normal, $A{_\tau\times_\sigma}B$ admits a quasitriangular structure if and only if there exist normalized elements $W=\sum W^1\otimes W^2\in B\otimes B$, $X=\sum X^1\otimes X^2\in A\otimes B$, $Y=\sum Y^1\otimes Y^2\in B\otimes A$ and $Z=\sum Z^1\otimes Z^2\in A\otimes A$ satisfying $C1'')-C19'')$ above. In this case, the quasitriangular structure is given as $\alpha=\sum Z^1{\bar{X}}^1X^1\otimes W^1Y^1\otimes Z^2\tensor{Y}{^2_\sigma}\epsilon_B(\tensor{\bar{X}}{^2_{\sigma}})\otimes X^2W^2$. 
	\end{coro}
	\begin{proof}
		Using $\sum1_{B\tau}\otimes a_\tau=1_B\otimes a$ in Theorem \ref{thm:3.7}.
	\end{proof}
	
	When $A$ and $B$ are Hopf algebras such that the smash biproduct bialgebra $A{_\tau\times_\sigma}B$ is also a Hopf algebra, this corollary is still true and it's equivalent to \cite[Theorem 2.7]{Z&W&J}. Let $A{_\tau\times_\sigma}B$ be a smash biproduct Hopf algebra for which $\sigma$ is right conormal and $\tau$ is left normal. They are exactly the bicrossproduct Hopf algebras (cf. \cite{Ma2}) by setting
	\begin{align*}
		\triangleright:B\otimes A\rightarrow A\quad&(b\otimes a\mapsto b\triangleright a=\sum a_{\sigma}\epsilon_B(b_{\sigma})),\\
		\rho:B\rightarrow B\otimes A\quad&(b\mapsto\sum b_{[0]}\otimes b_{[1]}=\sum b_{\tau}\otimes 1_{A\tau}).
	\end{align*}
	Then $A$ is a left $B$-module and $B$ is a right $A$-comodule. Moreover $A$ is an algebra in $_B\mathcal{M}$	and $B$ is a coalgebra in $\mathcal{M}^A$ and the conditions for a bicrossproduct Hopf algebra hold. The resulting bicrossproduct Hopf algebra $A{_{\triangleright}{\triangleright\hspace{-0.1em}\triangleleft}^\rho}B$ coincides with $A{_\tau\times_\sigma}B$.
	
	On the other hand, for a bicrossproduct Hopf algebra $A{_{\triangleright}{\triangleright\hspace{-0.1em}\triangleleft}^\rho}B$ of $A$ and $B$, by setting
	\begin{align*}
		&\sigma:B\otimes A\rightarrow A\otimes B\quad (b\otimes a\mapsto\sum b_{(1)}\triangleright a\otimes b_{(2)}),\\
		&\tau:A\otimes B\rightarrow B\otimes A\quad(a\otimes b\mapsto\sum b_{[0]}\otimes ab_{[1]}),
	\end{align*}
	we have a smash biproduct $A{_\tau\times_\sigma}B$ which coincides with the bicrossproduct. Thus these smash biproduct Hopf algebras satisfying both the left normal and right conormal conditions are exactly bicrossproduct Hopf algebras.
	
	The identities $C1'')-C19'')$ are also be classified into several groups in \cite{Z&W&J} and the reconstruction formula above becomes:
	\begin{align*}
		&\sum Z^1{\bar{X}}^1X^1\otimes W^1Y^1\otimes Z^2\tensor{Y}{^2_\sigma}\epsilon_B(\tensor{\bar{X}}{^2_{\sigma}})\otimes X^2W^2\\
		=&\sum Z^1{\bar{X}}^1X^1\otimes W^1Y^1\otimes Z^2(\tensor{\bar{X}}{^2}\triangleright \tensor{Y}{^2})\otimes X^2W^2.
	\end{align*}
	Hence Corollary \ref{thm:4.2} above generalizes \cite[Theorem 2.7]{Z&W&J} to bicrossproduct bialgebras.

	Another case is \cite{J}. The smash biproduct Hopf algebras studied in \cite{J} satisfy the conditions that $\sigma$ is right conormal and $\tau$ is left normal and right normal. Thus they can be treated as those in Section \ref{Sec:4.1} or in Section \ref{Sec:4.2} and the reconstruction formula will be more concise in form.

	\subsection{When $\sigma$ is left conormal}\label{Sec:4.3}
	Let $A{_\tau\times_\sigma}B$ be a smash biproduct bialgebra for which $\sigma$ is right conormal and additionally $\sigma$ is left conormal. Then $A$ and $B$ are both quotient-bialgebras of $A{_\tau\times_\sigma}B$.
	
	By applying $id_A\otimes\epsilon_B\otimes\epsilon_A\otimes id_B$ on both sides of $(B7)$, we have $\sum a_{\sigma}\otimes b_{\sigma}=a\otimes b$ for $\forall a\in A$, $b\in B$. Since $\sigma:B\otimes A\rightarrow A\otimes B$ is just the tensor flipping, these bialgebras shall be taken as the dual of the double crossproduct bialgebras (which can be not if $A$ or $B$ is infinite-dimensional).
	
	Since $\sigma$ is left conormal: $\sum a_{\sigma}\epsilon_B(b_{\sigma})=a\epsilon_B(b)$, the formula in Proposition \ref{thm:3.4} becomes:
	\begin{align*}
		&\sum \tensor{Z}{^1_{\tau_1\tau_2}}\tensor{\bar{X}}{^1_{\tau_3}}X^1\otimes W^1Y^1\otimes Z^2\epsilon_B(1_{B\tau_1\sigma_2}\tensor{\bar{X}}{^2_{\sigma_1}})\tensor{Y}{^2_{\sigma_1\sigma_2}}\otimes1_{B\tau_2}1_{B\tau_3}X^2W^2\\
		=&\sum \tensor{Z}{^1_{\tau_2}}\tensor{\bar{X}}{^1_{\tau_3}}X^1\otimes W^1Y^1\otimes Z^2\epsilon_B(\tensor{\bar{X}}{^2})\tensor{Y}{^2}\otimes1_{B\tau_2}1_{B\tau_3}X^2W^2\\
		=&\sum \tensor{Z}{^1_{\tau_2}}1_{A{\tau_3}}X^1\otimes W^1Y^1\otimes Z^2\tensor{Y}{^2}\otimes1_{B\tau_2}1_{B\tau_3}X^2W^2\\
		=&\sum \tensor{Z}{^1_\tau}X^1\otimes W^1Y^1\otimes Z^2\tensor{Y}{^2}\otimes1_{B\tau }X^2W^2.
	\end{align*}
	The identities in Proposition \ref{thm:3.5} become:
	\begin{align*}
		C1''')\quad&\sum \tensor{Z}{^1}a_{(1)}\otimes Z^2a_{(2)}=\sum a_{(2)}\tensor{Z}{^1}\otimes a_{(1)}\tensor{Z}{^2};\\
		C2''')\quad&\sum X^1a\otimes X^2=\sum a_\tau X^1\otimes1_{B\tau}\tensor{X}{^2};\\
		C3''')\quad&\sum Y^11_{B\tau}\otimes Y^2a_{\tau}=\sum Y^1\otimes aY^2;\\
		C4''')\quad&\sum Z^1\otimes Z^2\epsilon_B(b)\equiv\sum Z^1\otimes Z^2\epsilon_B(b);\\
		C5''')\quad&\sum X^1\otimes X^2b=\sum1_{A\tau}\tensor{X}{^1}\otimes b_{\tau}X^2;\\
		C6''')\quad&\sum Y^1b_\tau\otimes Y^21_{A\tau}=\sum bY^1\otimes \tensor{Y}{^2};\\
		C7''')\quad&\sum W^1b_{(1)}\otimes W^2b_{(2)}=\sum b_{(2)}W^1\otimes b_{(1)}W^2;\\
		C8''')\quad&\sum \tensor{Z}{^1_{(1)}}\otimes \tensor{Z}{^1_{(2)}}\otimes Z^2=\sum \tensor{Z}{^1}\otimes\bar{Z}^1\otimes Z^2\tensor{\bar{Z}}{^2};\\
		C9''')\quad&\sum \tensor{X}{^1_{(1)}}\otimes \tensor{X}{^1_{(2)}}\otimes X^2=\sum X^1\otimes{\bar{X}}^1\otimes X^2{\bar{X}}^2;\\
		C10''')\quad&\sum \tensor{Y}{^1_{\tau}}\otimes \tensor{Z}{^1_{\tau}}\otimes Z^2\tensor{Y}{^2}=\sum Y^1\otimes Z^1\otimes Y^2\tensor{Z}{^2};\\
		C11''')\quad&\sum \tensor{W}{^1_\tau}\otimes \tensor{X}{^1_\tau}\otimes X^2W^2=\sum W^1\otimes X^1\otimes W^2X^2;\\
		C12''')\quad&\sum \tensor{Y}{^1_{(1)}}\otimes \tensor{Y}{^1_{(2)}}\otimes Y^2=\sum Y^1\otimes {\bar{Y}}^1\otimes Y^2\tensor{\bar{Y}}{^2};\\
		C13''')\quad&\sum \tensor{W}{^1_{(1)}}\otimes \tensor{W}{^1_{(2)}}\otimes W^2=\sum W^1\otimes{\bar{W}}^1\otimes W^2{\bar{W}}^2;\\
		C14''')\quad&\sum  Z^1\otimes\tensor{Z}{^2_{(1)}}\otimes \tensor{Z}{^2_{(2)}}=\sum \tensor{Z}{^1}\tensor{\bar{Z}}{^1}\otimes{\bar{Z}}^2\otimes Z^2;\\
		C15''')\quad&\sum \tensor{Z}{^1_{\tau_1}} X^1\otimes (1_{B\tau_1}\tensor{X}{^2})_{\tau_2}\otimes\tensor{Z}{^2_{\tau_2}}=\sum \tensor{Z}{^1}\tensor{X}{^1}\otimes X^2\otimes Z^2;\\
		C16''')\quad&\sum X^1\otimes\tensor{X}{^2_{(1)}}\otimes \tensor{X}{^2_{(2)}}=\sum X^1\tensor{\bar{X}}{^1} \otimes{\bar{X}}^2\otimes X^2;\\
		C17''')\quad&\sum  Y^1\otimes\tensor{Y}{^2_{(1)}}\otimes \tensor{Y}{^2_{(2)}}=\sum Y^1{\bar{Y}}^1\otimes {\bar{Y}}^2\otimes Y^2;\\
		C18''')\quad&\sum W^1Y^1\otimes\tensor{W}{^2_\tau}\otimes \tensor{Y}{^2_\tau}=\sum Y^1W^1 \otimes W^2\otimes Y^2;\\
		C19''')\quad&\sum  W^1\otimes\tensor{W}{^2_{(1)}}\otimes \tensor{W}{^2_{(2)}}=\sum W^1{\bar{W}}^1 \otimes {\bar{W}}^2\otimes W^2.
	\end{align*}
	
	Here $C4''')$ is an identity. But for consistence in form, we still keep it.

	Then Theorem \ref{thm:3.7} gives the following:
	\begin{coro}\label{thm:4.3}
		Let $A{_\tau\times_\sigma}B$ be a smash biproduct bialgebra. Under conditions that $\sigma$ is left conormal and right conormal, $A{_\tau\times_\sigma}B$ admits a quasitriangular structure if and only if there exist normalized elements $W=\sum W^1\otimes W^2\in B\otimes B$, $X=\sum X^1\otimes X^2\in A\otimes B$, $Y=\sum Y^1\otimes Y^2\in B\otimes A$ and $Z=\sum Z^1\otimes Z^2\in A\otimes A$ satisfying $C1''')-C19''')$ above. In this case, the quasitriangular structure is given as $\alpha=\sum \tensor{Z}{^1_\tau}X^1\otimes W^1Y^1\otimes Z^2\tensor{Y}{^2}\otimes1_{B\tau }X^2W^2$.
	\end{coro}
	\begin{proof}
		Using $\sum a_{\sigma}\epsilon_B(b_{\sigma})=a\epsilon_B(b)$ in Theorem \ref{thm:3.7}.
	\end{proof}
	
	When $A$ and $B$ are Hopf algebras such that the smash biproduct bialgebra $A{_\tau\times_\sigma}B$ is also a Hopf algebra, this corollary is still true and it's \cite[Theorem 3.1]{J2}. By the conditions that $\sigma$ is left conormal and right conormal, $\sigma$ is just the tensor flipping. Thus the smash biproduct Hopf algebra $A{_\tau\times_\sigma}B$ is exactly the Hopf algebra $A{_{\omega}{\triangleright\hspace{-0.1em}\triangleleft}}B$ studied in \cite{J2}. And we have
	\begin{align*}
		&\sum\tensor{Z}{^1_\tau}X^1\otimes W^1Y^1\otimes Z^2\tensor{Y}{^2}\otimes1_{B\tau}X^2W^2\\
		\stackrel{C2''')}{=}&\sum X^1\tensor{Z}{^1}\otimes W^1Y^1\otimes Z^2\tensor{Y}{^2}\otimes X^2W^2.
	\end{align*}
	which is the formula in \cite{J2}.
	
	\begin{rmk}
		The smash biproduct bialgebras considered in Corollary \ref{thm:4.1}, \ref{thm:4.2} and \ref{thm:4.3} are distinct. Thus these results are also distinct factorization theories. On the other hand, there do exist smash biproduct bialgebras for which only $\sigma$ is right conormal. Thus our theory (mainly Theorem \ref{thm:3.7}) is not just the “union” of  Corollary \ref{thm:4.1}, \ref{thm:4.2} and \ref{thm:4.3}. Our theory generalizes them and thus generalizes the results of \cite{J,J2,M&W,Z&W&J,Z&Z}.
	\end{rmk}
	
	\section*{Acknowledgement}
	
	The author would like to thank Professor Shenglin Zhu for his useful supports and discussions.

	\end{document}